\documentclass[11pt]{amsart}
\scrollmode
\usepackage[utf8]{inputenc}
\usepackage{setspace}
\usepackage[margin=.9in]{geometry}
\usepackage{geometry, url}

\usepackage{amsmath,amsthm,amssymb,amsfonts,xcolor}
\usepackage{graphicx, caption, subcaption, listings, float}
\usepackage{multirow,enumitem, hyperref, tabularx}
\usepackage{array,algorithm, algpseudocode}
\usepackage{makecell}
\usepackage{enumitem}

\usepackage{xcolor, mathtools}
\usepackage{tabularx,colortbl, array}

\newcommand{\y}{y}

\graphicspath{{figures/}}
\usepackage{comment}
\theoremstyle{plain}
\newtheorem{thm}{Theorem}[section]
\newtheorem{cor}[thm]{Corollary}
\newtheorem{lem}[thm]{Lemma}
\newtheorem{prop}[thm]{Proposition}

\newtheorem{defn}[thm]{Definition}
\newtheorem{rem}[thm]{Remark}

\newtheorem{exm}[thm]{Example}
\newtheorem{prob}[thm]{Problem}
\usepackage{placeins}
\newcommand{\icol}[1]{% inline column vector
  \left(\begin{smallmatrix}#1\end{smallmatrix}\right)%
}

\title{Semi-discrete multi-to -one dimensional variational problems}
\date{\today}
\author{Omar Abdul Halim} 
\author{Daniyar Omarov}
\author{Brendan Pass}
\email{oabdulh1@ualberta.ca, daniyar@ualberta.ca, pass@ualberta.ca}

\thanks{The work of OAH was completed in partial fulfillment of the requirements for a doctoral degree in
mathematics at the University of Alberta. BP is pleased to acknowledge the support of Natural Sciences and
Engineering Research Council of Canada Discovery Grant number 04864-2024. DO gratefully acknowledge that this research was supported in part by the Pacific Institute for the Mathematical Sciences.}
\address{University of Alberta, Edmonton, Alberta, Canada}
\begin{document}
\singlespacing
\begin{abstract}
We study a class of semi-discrete variational problems that arise in economic matching and game theory, where agents with continuous attributes are matched to a finite set of outcomes with a one dimensional structure. 
Such problems appear in applications including Cournot-Nash equilibria, and hedonic pricing,   and can be formulated as problems involving optimal transport between spaces of unequal dimensions.  In our discrete strategy space setting, we establish analogues of results developed for a continuum of strategies in \cite{nenna2020variational}, ensuring solutions have a particularly simple structure under certain conditions.  This has important numerical consequences, as it is natural to discretize when numerically computing solutions. We leverage our results to develop efficient algorithms for these problems  which  scale significantly better than standard optimal transport solvers, particularly when the number of discrete outcomes is large, provided our conditions are satisfied. We also establish rigorous convergence guarantees for these algorithms.  We illustrate the advantages of our approach by solving a range of numerical examples; in many of these our new solvers outperform alternatives by  a considerable margin.

%We extend the notion of nestedness—a structural condition that ensures regularity and simplifies computation in continuous one-dimensional target settings—to the semi-discrete case. We derive sufficient conditions under which the solutions of a class of variational problems are discretely nested. Assuming the solution is nested, we develop  algorithms that scale significantly better than standard optimal transport solvers, particularly when the number of discrete outcomes is large when the solution exhibits nestedness. Our approach retains the theoretical appeal of continuous models while offering computational advantages for large-scale semi-discrete matching problems.
\end{abstract}

\large\maketitle

\section{Introduction}
A broad class of problems in economics, game theory, and urban planning can be formulated as variational problems on the space of probability measures involving optimal transport.  More explicitly, they amount minimizing a functional of the form:

\begin{equation}\label{eqn: variational problem}
\min_{\nu \in \mathcal{P}(Y)} \left\{ \mathcal{W}_c(\mu, \nu) + \mathcal{F}(\nu) \right\},
\end{equation}
where $\mu$ is a fixed distribution of agents, $\nu$ a free distribution of strategies, $\mathcal{W}_c$ is the cost of optimal transport between them with cost function $c$ and $\mathcal{F}$ a functional on the space $\mathcal{P}(Y)$ of probability measures, with different forms depending on the precise problem under consideration.  Examples include  \emph{Cournot-Nash equilibria} models~\cite{blanchet2014Nash,blanchet2014remarks, blanchet2016optimal,BlanchetCarlierNenna2017}, where agents interact strategically in markets with congestion or competition effects; \emph{urban planning} problems~\cite{AguehCarlier2011,buttazzo2005model,buttazzo2009mass,CarlierEkeland2004,carlier2005urban}, which involve the optimal spatial allocation of infrastructure or public goods; and \emph{hedonic pricing} models~\cite{chiappori2010hedonic,COCV_2005__11_1_57_0}, where agents evaluate  goods or services based on individual preferences. %Many of these models can be formulated as variational problems involving optimal transport; more precisely, they take the form

  These problems are computationally challenging due to the complexity of optimal transport.  However, in at least one important regime they can be simplified considerably.  It is often reasonable to parametrize agents by a distribution $\mu$ on a  high dimensional space $X$ (reflecting a high degree of heterogeneity), whereas available strategies have more limited variability and are therefore modeled by a one-dimensional space $Y$.  The corresponding \emph{multi-to one-dimensional} optimal transport problem admits a greatly simplified characterization of solutions, provided that a certain condition on $c$, $\mu$ and $\nu$, recently developed in \cite{chiappori2017multi} and known as \emph{nestedness}, is satisfied.  Furthermore, hypotheses on $\mathcal{F}$, $c$ and $\mu$ have recently been developed ensuring that for the optimal $\nu$ in \eqref{eqn: variational problem}, the optimal transport problem arising there is nested \cite{nenna2020variational}. 

These advances greatly enhance the theoretical tractability  of \eqref{eqn: variational problem}, but a significant issue arises when trying to apply them numerically. In practice, when computing solutions, one typically discretizes the target space, so that the one-dimensional space $Y$ is replaced by a finite set $Y_N$ and the unknown measure $\nu \in \mathcal{P}(Y_N)$ becomes discrete.   However, the original notion of nestedness relies on the continuous structure of $Y$ in a crucial way, and so the simple characterization of solutions in \cite{nenna2020variational} cannot be directly exploited numerically.\footnote{Some related computations with discrete measures are carried out in \cite{nenna2024note}.  Calculations there are done using variables in a continuous ambient space $Y$, and the effect on the solution of passing back and forth between discrete and continuous problems is neglected.}

On the other hand, an analogue of nestedness for optimal transport problems with discrete target spaces has been recently introduced in \cite{halim2025multitoonedimensionalscreeningsemidiscrete}, together with an essentially closed form characterization of solutions when this condition is satisfied.  The purpose of this paper is twofold.  First, to translate the analysis in \cite{nenna2020variational} to the discrete target setting, and establish conditions on $\mu$, $c$, $F$, and $Y_N$ under which solutions to \eqref{eqn: variational problem} satisfy discrete nestedness, and, secondly, to exploit these results to develop efficient numerical algorithms.

Our focus here is therefore on discrete versions of the two main types of functionals $\mathcal{F}$ studied in \cite{nenna2020variational}: \emph{internal energies}, or \emph{congestion terms}, modeling situations where agents prefer strategies which differ from those chosen by other agents, and those modeling \emph{hedonic pricing problems}, in which both buyers and sellers seek to maximize their utilities while exchanging goods chosen from a set $Y_N$ of feasible goods or contracts; in this case, the functional $\mathcal{F}$ is itself the optimal transport distance to another fixed distribution $\mu_1$.

The paper is organized as follows. In Section~\ref{nestedness}, we review relevant background on optimal transport and discrete nestedness. Section~\ref{internal energy} considers the semi-discrete framework of the variational problem with a congestion term and introduces a numerical method for computing the solution. In Section~\ref{hedonic problem}, we present the hedonic pricing problem in the semi-discrete setting and propose an algorithm that utilizes the nested structure of the solution to improve computational efficiency in the case of nested solution.

\section{Nestedness of optimal transport}\label{nestedness}
\subsection{Optimal transport}
Here we briefly introduce the optimal transport problem; a fuller introduction can be found in, for example, \cite{santambrogio2015optimal,Villani2009}.  Given probability measures $\mu$ and $\nu$ on bounded sets $X \subset \mathbb{R}^{d_X}$ and  $Y\subseteq \mathbb{R}^{d_Y}$, respectively, and a cost function $c \in C(X \times Y)$, the optimal transport problem is to minimize

%Optimal transport theory addresses the problem of reallocating a probability measure $\mu$ on a domain $X \subset \mathbb{R}^d$ to another measure $\nu$ on $Y \subset \mathbb{R}^d$ in a way that minimizes an aggregate cost defined by a function $c \in C(X \times Y)$. 

%In the Kantorovich relaxation, one seeks a coupling measure $\gamma$ on $X \times Y$ with marginals $\mu$ and $\nu$ that minimizes the total transportation cost:
\begin{equation}\label{eqn: OT problem}
\mathcal{W}_c(\mu,\nu) = \min_{\gamma \in \Gamma(\mu,\nu)} \int_{X \times Y} c(x,y) \, d\gamma(x,y),
\end{equation}
where $\Gamma(\mu, \nu)$ consists of all probability measures $\gamma$ on $X \times Y$ with marginals $\mu$ and  $\nu$.

As a linear program, \eqref{eqn: OT problem} admits a dual problem and it is well known that strong duality holds; that is,

%To better understand the structure of optimal transport plans, we consider its dual formulation which can be written as:
\begin{equation}\label{eqn: OT dual}
\mathcal{W}_c(\mu,\nu) = \sup_{(u,v) \in \mathcal{V}} \left\{ \int_X u(x) d\mu(x) + \int_Y v(y) d\nu(y) \right\},
\end{equation}
where $\mathcal{V} = \{(u,v) \in L^1(\mu) \times L^1(\nu) \mid u(x) + v(y) \leq c(x,y), \ \forall (x,y) \in X \times Y \}$ is the set of feasible potentials.  Moreover, it is well known that a minimizer $\gamma$ in \eqref{eqn: OT problem} and a maximizer $(u,v)$ in \eqref{eqn: OT dual} both exist, and that $u(x) + v(y) = c(x,y)$, $\gamma$ almost everywhere.  Furthermore, the optimal dual potentials can be chosen to be \emph{$c$-concave}, meaning that $u(x) = \inf_{y \in Y} \{ c(x,y) - v(y) \}$ and  $v(y) = \inf_{x \in X} \{ c(x,y) - u(x)\}$.

We are especially interested here in the semi-discrete setting, where the source measure $ \mu $ is continuous, while the target measure  $ \nu = \sum_{i=1}^N \nu_i \delta_{y_i} $ is supported on finitely many points $ \{y_i\}_{i=1}^N $ (see \cite{Galichon2016,merigot2021optimal,PeyreCuturi2019} for a detailed review). In this case, the optimal dual potential $v=(v_1,...,v_N)$ becomes a vector in $\mathbb{R}^N$ where $v_i=v(y_i).$ It induces the measurable sets 
\[
X_i := \left\{ x \in X \mid c(x,y_i) - v_i \leq c(x,y_j) - v_j \text{ for all } j \right\}.
\]
Any optimal measure $\gamma$ can only transport $x$ to $y_i$ if $x\in X_i$, and the dual potential $ u $ is piecewise defined on these sets by $u(x) = c(x, y_i) - v_i \quad \text{for } x \in X_i.$
On the boundary between regions $ X_i $ and $ X_j $, the potential satisfies $c(x, y_i) - v_i = c(x, y_j) - v_j,$ which determines the geometry of the transport cells and the interfaces between them.

\subsection{A sufficient condition for discrete nestedness}
 Throughout the rest of the paper, we specialize to the case where $d_Y=1$ and simply denote $d_X=d$. Let $X\subset \mathbb{R}^d$ and $Y=\{y_1,\dots, y_N\}\subset \mathbb{R}$ %\{\textcolor{blue}{DO: is it $\mathbb{R}$ or $\mathbb{R}^d$? Above we have used $\mathbb{R}^d$ for both  measures.} \}
be bounded sets such that $X$ is open. We will make the following assumption, which is a slight weakening of the well known twist condition, on $c\in C^2(X\times Y):$   
\begin{enumerate}[label={(H\arabic*)}]
%\item\label{nonzero}$c_{yx}(x,y)\neq 0 \text{ for all $(x,y)\in X\times Y$ and}$
\item\label{nonzero1} $D_xc(x,y_{i+1}) -D_xc(x,y_{i}) \neq 0 \text{ for all $x\in X$ and $1\le i\le N-1.$}$

\end{enumerate}

Let $\mu \in \mathcal{P}(X)$ , and let $\nu \in \mathcal{P}(Y)$ be given by $\nu = \sum_{i=1}^N \nu_i \delta_{y_i} ,$ where $\delta_{y_i}$ is the Dirac mass at $y_i \in Y$ and $\nu_i \in [0,1]$.
 We define the following level and sub-level sets:
\[ X^N_{=}(y_i,k):=\{x\in X\,:\, c(x,y_{i+1})-c(x,y_{i})=k\},\]
\[X_{\ge}^N(y_i,k):=\{x\in X\,:\, c(x,y_{i+1})-c(x,y_{i})\ge k\},\]
and $X_{>}^N(y,k):=X_{\ge}^N(y,k)\setminus X_{=}^N(y,k).$ For the rest of the paper we assume that  $\mu(X_=^N(y_i,k)) = 0$ for all $y_i$ and $k.$

\begin{rem}
    By condition \ref{nonzero1} and using the Implicit Function Theorem on the equation $c(x,y_{i+1})-c(x,y_{i})=k$ for some constant $k,$ we get that $X_=^N(y_i,k)$ has codimension $1.$  In particular, standard regularity assumptions on $\mu$ then imply our hypothesis $\mu(X_=^N(y_i,k)) = 0$.
\end{rem}

\begin{defn}
     We say that the optimal transport problem $(c,\mu,\nu)$ is discretely nested if for all $1\le i\le N-2$ such that $\nu_{i+1}>0$, we have 
    \[ X_{\ge}^N(y_{i},k^N(y_{i}))\subset X_{>}^N(y_{i+1},k^N(y_{i+1})),\] 
    where $k^N(y_{r})$ is chosen such that $\mu(X_{\ge}^N(y_{r},k^N(y_{r})))=\sum_{i=1}^r\nu_i.$
\end{defn}
Note that this definition is equivalent to the definition of discrete nestedness introduced in \cite{halim2025multitoonedimensionalscreeningsemidiscrete}, where, whenever $\sum_{r=i+1}^j \nu_r > 0$, we require $X_{\ge}^N(y_{i},k^N(y_{i})) \subset X_{>}^N(y_{j},k^N(y_{j})).$

The discrete nestedness condition imposes a monotonicity structure on the family of superlevel sets $ X_{\ge}^N(y_i, k^N(y_i)) $. This structure ensures that the regions associated with successive outcomes are ordered in a way that enables a piecewise construction of the optimal transport map.  The following result provides a precise characterization of the solution under this condition.  It is proven in \cite{halim2025multitoonedimensionalscreeningsemidiscrete}, and can be seen as a semi-discrete version of Theorem 4 in \cite{chiappori2017multi}.

\begin{thm}[\cite{halim2025multitoonedimensionalscreeningsemidiscrete}]\label{nestchar}
Assume the optimal transport problem $(c,\mu, \nu)$ is discretely nested. Then, setting $X_1=X^N_{>}(y_1, k^N(y_1)),$ $X_N=X\setminus X^N_{\ge}(y_{N-1}, k^N(y_{N-1})),$ and $X_i = X^N_{>}(y_i, k^N(y_i)) \setminus X^N_{\geq}(y_{i-1}, k^N(y_{i-1}))$ for all $1<i<N$, the optimal transport plan is unique and pairs each $x \in X_i$ with $y_i$ for all $1\leq i \leq N$.  Furthermore, the optimal potentials $(u,v)$ are defined by $u(x) = c(x, y_i) - v_i$ for all $x \in X_i$,  where 
\[
v_i = \sum_{j=1}^{i-1}k^N(y_j),
\]
for $1<i\le N$ where $v_1=0.$ 
\end{thm}

We now turn to the task of establishing sufficient conditions for discrete nestedness.   Our work in the remainder of this section mirrors the theory developed for continuous one-dimensional targets in Section 2.2 of \cite{nenna2020variational}. For $1\le i\le N-2,$ we define the minimal mass difference as follows,

\[D^{\min}_{\mu}(y_i,k_i)=\mu\left(X_\ge ^N(y_{i+1},k_{\max}(y_i,k_i))\setminus X_\ge ^N(y_{i},k_i) \right)\] where $k_{\max}(y_i,k_i)=\sup\left\{k\in \mathbb{R}:\, X_\ge^N(y_i,k_i)\subseteq X_\ge^N(y_{i+1},k)\right\}$ for some  $k_i.$  This minimal mass difference tells us the least additional $\mu$-mass required so that a superlevel set $ X_{\ge}^N(y_i, k_i) $  becomes contained in a corresponding set for $ y_{i+1} $. This leads to a necessary condition for discrete nestedness, as well as a sufficient condition under strict inequality.

\begin{thm}\label{nested theorem}
     If $(c,\mu,\nu)$ is discretely nested, then $D^{\min}_{\mu}(y_i,k^N(y_i))\le\nu_{i+1}$ for all $1\le i\le N-2.$ Conversely, if $D^{\min}_{\mu}(y_i,k^N(y_i))<\nu_{i+1}$ for all $1\le i\le N-2,$ then $(c,\mu,\nu)$ is discretely nested.
\end{thm}
%{\color{blue}BP: The proof of this theorem and Proposition 3.1 are not too long.  I'd probably just keep them in the body rather than create an appendix for them.}
\begin{proof}
    Assume that $(c,\mu,\nu) $ is discretely nested. Then, for all $1\le i\le N-2$ we have $X_{\ge}^N(y_{i},k^N(y_{i}))\subset X_{>}^N(y_{i+1},k^N(y_{i+1}))\subset X_{\ge}^N(y_{i+1},k^N(y_{i+1})) $ which implies $k^N(y_{i+1})\le k_{\max}(y_i,k^N(y_i)).$ Hence,
    \[\begin{array}{ll}
    D^{\min}_{\mu}(y_i,k^N(y_i))&=\mu(X_\ge ^N(y_{i+1},k_{\max}(y_i,k^N(y_i)))\setminus X_\ge ^N(y_{i},k^N(y_i)))\vspace{1pt}\\
    &\le\mu(X_\ge ^N(y_{i+1},k^N(y_{i+1}))\setminus X_\ge ^N(y_{i},k^N(y_i))\vspace{1pt}\\
    &=\mu(X^N_{\ge}(y_{i+1},k^N(y_{i+1})))-\mu(X^N_{\ge}(y_{i},k^N(y_{i})))=\nu_{i+1}
    \end{array}\]
    for all $1\le i\le N-2.$

    Assume that $D^{\min}_{\mu}(y_i,k^N(y_i))<\nu_{i+1}$ for all $1\le i\le N-2.$ If $(c,\mu,\nu)$ is not discretely nested, then there exists $1\le j\le N-2$ such that $X_{\ge}^N(y_{j},k^N(y_{j}))\not\subset X_{>}^N(y_{j+1},k^N(y_{j+1})).$ Thus, $k^N(y_{j+1})\ge k_{\max}(y_j,k^N(y_j))$ and
    \[\mu(X_\ge ^N(y_{j+1},k^N(y_{j+1}))\setminus X_\ge ^N(y_{j},k^N(y_j)))\le D^{\min}_{\mu}(y_j,k^N(y_j)).\] But,
   \[\begin{array}{ll}
   \nu_{j+1}&=\mu(X^N_{\ge}(y_{j+1},k^N(y_{j+1})))-\mu(X^N_{\ge}(y_{j},k^N(y_{j})))\vspace{1pt}\\
   &\le \mu(X_\ge ^N(y_{j+1},k^N(y_{j+1}))\setminus X_\ge ^N(y_{j},k^N(y_j))\le D^{\min}_{\mu}(y_j,k^N(y_j))
   \end{array}\]
   which is a contradiction.
\end{proof}
The minimal mass condition in Theorem~\ref{nested theorem}  yields a verifiable sufficient condition for discrete nestedness that does not require prior knowledge of the splitting levels $k^N(y_i)$. In particular, by uniformly bounding the minimal mass difference across all splitting levels, we obtain the following corollary.

\begin{cor}\label{lowerbound}
    If for all $1\le i\le N-2$ we have  $\sup_{k\in \mathbb{R}}D^{\min}_{\mu}(y_i,k)-\nu_{i+1}<0,$
    then $(c,\mu,\nu)$ is discretely nested.
\end{cor}
\begin{proof}
    The condition $\sup_{k\in \mathbb{R}}D^{\min}_{\mu}(y_i,k)-\nu_{i+1}<0$ implies that $D^{\min}_{\mu}(y_i,k)-\nu_{i+1}<0$ for all $k$ in particular 
    \(D^{\min}_{\mu}(y_i,k^N(y_i))-\nu_{i+1}<0\) for all $1\le i \le N-2$ and by Theorem \ref{nested theorem} we get discrete nestedness.
\end{proof}

\begin{exm}  \label{example1}
Let $\mu$ be the uniform measure on $X = (0,1)^2$ with $c(x, y_i) = -x_1 y_i - x_2 F(y_i)$ for all  $x=(x_1,x_2)\in X$ and $y_i \in Y$, where $F$ is an increasing convex function. In this case, the set $X_=^N(y_i, k)$ is a line orthogonal to the vector $\langle y_{i+1} - y_i, F(y_{i+1}) - F(y_i) \rangle$.  

For a fixed $k_0$, the set $X_{\ge}^N(y_{i+1}, k_{\max}(y_i, k_0)) \setminus X_{\ge}^N(y_i, k_0)$ is contained within a triangle whose vertices are given by the intersections of $X_=^N(y_i, k_0)$, $X_=^N(y_{i+1}, k_{\max}(y_i, k_0))$, and the $x_1$-axis. It follows that  
\[
D_\mu^{\min}(y_i, k_0) \leq \sup_{k \in \mathbb{R}} D_\mu^{\min}(y_i, k) \le \frac{1}{2} \left( \frac{F(y_{i+2}) - F(y_{i+1})}{y_{i+2} - y_{i+1}} - \frac{F(y_{i+1}) - F(y_i)}{y_{i+1} - y_i} \right).
\]  
This supremum is bounded above by the area of the triangle formed by the intersection of $X_=^N(y_i, k)$, $X_=^N(y_{i+1}, k_{\max}(y_i, k))$, and the $x_1$-axis for some $k$ such that the intersection of $X_=^N(y_i, k)$ and $X_=^N(y_{i+1}, k_{\max}(y_i, k))$ lies on the line $x_2 = 1$.  

By Corollary \ref{lowerbound}, the problem $(c,\mu,\nu)$ is discretely nested whenever 
\[\sup_{k \in \mathbb{R}} D_\mu^{\min}(y_i, k)\le\frac{1}{2} \left( \frac{F(y_{i+2}) - F(y_{i+1})}{y_{i+2} - y_{i+1}} - \frac{F(y_{i+1}) - F(y_i)}{y_{i+1} - y_i} \right)<\nu_{i+1}\] for all $1\le i\le N-2.$
\end{exm}

\section{Internal energy}\label{internal energy}

As above, we take $Y=\{y_1,...,y_N\}$ to be discrete throughout this section.  We consider variational problems of the form
\begin{equation}\label{cong}
\min_{\nu \in \mathcal{P}(Y)} \left\{ \mathcal{W}_c(\mu, \nu) + \mathcal{F}(\nu) \right\},
\end{equation}
for a fixed $ \mu \in \mathcal{P}(X) $, where $ \mathcal{F}(\nu) = \sum_{i=1}^N f(\nu_i)\nu_i $ and $ f : [0, \infty) \to \mathbb{R} $ is continuously differentiable on $ (0, \infty) $, strictly convex, has superlinear growth, and satisfies $ \lim_{s \to 0^+} f'(s) = -\infty $. These conditions on $ f $ guarantee the existence  and uniqueness  of a minimizer. Moreover, the strict convexity of $ f $ implies that for each fixed $v$ the function
\(
s \mapsto \sum_{i=1}^N (f')^{-1}(s - v)
\)
is strictly increasing, and therefore admits an inverse, which we denote by $ J_v $.
Our work in this section closely follows the corresponding analysis in the continuous case in Section 3.1 of \cite{nenna2020variational}.

\subsection{Discrete nestedness of the solutions}
To proceed with our analysis, we first obtain  bounds on the optimal weights $\nu_i$. In particular, the following proposition provides lower and upper estimates on each $\nu_i$ in terms of $f,$ the geometry of  $Y$ and the cost function $c$.

\begin{prop}\label{bounds}
    The minimizer $\nu$ of \eqref{cong} satisfies
    \[(f')^{-1}(J_{-M_c|y_i-y_1|}(1)-M_c|y_i-y_1|)\le \nu_i\le (f')^{-1}(J_{M_c|y_i-y_1|}(1)+M_c|y_i-y_1|)\] where $M_c=\sup_{(x,y_j,y_k)}\frac{|c(x,y_j)-c(x,y_k)|}{|(x,y_j)-(x,y_k)|}$ and $|\cdot|$ is the Euclidean norm.
\end{prop}
\begin{proof}
    Let $\nu$ be the minimizer of \eqref{cong} and let $v=(v_i)_{i=1}^N$ such that $(u,v)$ is the solution to the  Kantorovich dual problem \eqref{eqn: OT dual} between the source measure $\mu$ on $X$ and the target measure $\nu = \sum_{i=1}^N \nu_i \delta_{y_i}$.

%{\color{red}BP: Why do we bother extending $v$ to the continuous region? Why not just work on the discrete set?}

Without loss of generality, we normalize by taking $v_1 = 0$.
 From \cite{blanchet2016optimal}, we get the first order optimality condition of \eqref{cong} which is 
    \begin{equation}\label{optcon}
        v_i+f'(\nu_i)=C,
    \end{equation} for some constant $C$.    From \eqref{optcon}, we get  $\nu_i=(f')^{-1}(C-v_i)$ and using the fact that $\nu$ is a probability measure we deduce $1=\sum_{i=1}^N\nu_i=\sum_{i=1}^N(f')^{-1}(C-v_i)$  and so $J_v(1)=C.$

    From \cite{McCann2001} we know that $v$ is a Lipschitz function with the constant $M_c=\sup_{(x,y_j,y_k)}\frac{|c(x,y_j)-c(x,y_k)|}{|(x,y_j)-(x,y_k)|}$ where $|\cdot|$ is the Euclidean norm, thus $-M_c|y_i-y_1|\le v_i\le M_c|y_i-y_1|.$ By the monotonicity of $(f')^{-1},$ we get 
    \[\sum_{i=1}^N(f')^{-1}(C-M_c|y_i-y_1|)\le \sum_{i=1}^N(f')^{-1}(C-v_i)=1\le \sum_{i=1}^N(f')^{-1}(C+M_c|y_i-y_1|).\] Apply $J_{-M_c|y_i-y_1|}$ to $1\le\sum_{i=1}^N(f')^{-1}(C+M_c|y_i-y_1|)$ to get 
    $J_{-M_c|y_i-y_1|}(1)\le C.$ Similarly, we get  $C\le J_{M_c|y_i-y_1|}(1).$
    Therefore, 
    \[(f')^{-1}(J_{-M_c|y_i-y_1|}(1)-M_c|y_i-y_1|)\le (f')^{-1}(C-v_i)=\nu_i\le (f')^{-1}(J_{M_c|y_i-y_1|}(1)+M_c|y_i-y_1|).\]
\end{proof}
Using the lower bound on the components of $\nu$ from Proposition~\ref{bounds}, we now derive a concrete sufficient condition ensuring that  $(c,\mu,\nu)$ is discretely nested. This is done by inserting the lower estimate into the mass comparison criterion of Corollary~\ref{lowerbound}.

\begin{thm}\label{suffcon}

Assume that 
    \[\sup_{k\in\mathbb{R}} D_{\mu}^{\min}(y_i,k)-(f')^{-1}(J_{-M_c|y_{i+1}-y_1|}(1)-M_c|y_{i+1}-y_1|)<0\] for all $1\le i\le N-2.$
  Then, letting $\nu$ be the minimizer of \eqref{cong}, $(c,\mu,\nu )$ is discretely nested.
\end{thm}
\begin{proof}
   The proof follows from Corollary \ref{lowerbound} and Proposition \ref{bounds}.
\end{proof}
The following example illustrates how the preceding result can be used to ensure nestedness of the solution on a concrete example.

\begin{exm}\label{example suffic}

    Going back to Example \ref{example1}, when  $\mu$ is the uniform measure on $X = (0,1)^2$ with $c(x, y_i) = -x_1 y_i - x_2 F(y_i)$ for all $y_i \in Y$ with $0=y_1\le y_i<y_{i+1}$, and $F$ is an increasing convex function, let $f(s)=s\ln(s).$ In this case, apart from an irrelevant additive constant we get $f'(s)=\ln(s)$ and $J_v$ is the inverse of $s\mapsto \sum_{i=1}^Ne^{s-v_i}=e^s\sum_{i=1}^Ne^{-v_i}$ where $v_i$ is the Kantorovich potential. Hence, $J_v(1)=\ln\left(\frac{1}{\sum_{i=1}^Ne^{-v_i}}\right).$ Also, we have  $M_c=1,$ and by Proposition \ref{bounds}, we get 
    \[\nu_{i+1}\ge e^{\ln\left(\frac{1}{\sum_{p=1}^Ne^{y_{i+1}}}\right)-y_{i+1}}=\frac{e^{-y_{i+1}}}{Ne^{y_{i+1}}}.\] By Theorem \ref{suffcon}, the problem $(c,\mu,\nu)$ is nested where $\nu$ is the minimizer whenever
    \[\begin{array}{ll}
    \sup_{k\in \mathbb{R}}D_\mu^{\min}(y_i,k)\le\frac{1}{2} \left( \frac{F(y_{i+2}) - F(y_{i+1})}{y_{i+2} - y_{i+1}} - \frac{F(y_{i+1}) - F(y_i)}{y_{i+1} - y_i} \right)< \frac{1}{Ne^{2y_{i+1}}}
    \end{array}\]
    for all $1\le i\le N-2.$ When $F(y)=\frac{y^2}{A}$ for some constant $A$ and $y_i=\frac{i}{N},$ we have $\sup_{k\in \mathbb{R}}D_\mu^{\min}(y_i,k)\le\frac{1}{AN}$ which implies that  the problem $(c,\mu,\nu)$ is discretely nested when $A>e^{2}.$  
\end{exm}

 The structure described above admits a meaningful interpretation in practical decision-making contexts. Inspired by the holiday choice in Section 2.1 of \cite{blanchet2014Nash}, consider, for instance, a population of individuals (represented by the measure $\mu$) choosing among a finite set of national parks, indexed by $y_i \in Y$. Each park is characterized by a scalar attribute $y_i$, which quantifies its scenic quality—reflecting features such as landscape diversity, iconic viewpoints, and biodiversity. We can interpret $F(y_i)$ as the (normalized) quality of visitor-facing utilities and services at the park, such as accessibility of trailheads, availability and reliability of shuttle services, cleanliness and density of restrooms, potable water points, and safety infrastructure. Empirically, these utilities tend to improve with scenic quality, and often do so in an accelerating fashion—top-tier scenic parks attract disproportionately larger investments in visitor infrastructure.

Each individual is described by a vector $x = (x_1, x_2)$, where $x_1$ represents their preference for scenic quality, and $x_2$ encodes their preference for high-quality visitor utilities. The cost function $c(x, y_i) = -x_1 y_i - x_2 F(y_i)$ then encodes preferences that weigh the park’s natural scenery (through $y_i$) against the quality of its visitor amenities (through $F(y_i)$). The allocation of individuals across parks arises as a Cournot--Nash equilibrium~\cite{blanchet2016optimal}: each visitor chooses the park that maximizes their own utility while considering the overall distribution of other visitors. Congestion effects, modeled by the term $f(s) = s \ln s$, capture the population’s aversion to overcrowding—visitors prefer not to choose parks that are already popular and heavily frequented by others.

     %{\color{red} BP: Shouldn't the cost increase with price?  Perhaps $-x_2$ should quantify aversion to price?  Also, the assumption below about how agents distribute themselves isn't really an assumption is it?  Isn't it known to be equivalent to an equilibrium by Blanchet-Carlier?}

%Note that a similar analysis to Example \ref{example suffic} applies when $F$ is a decreasing convex function (for instance, when $y_i$ represents the distance from a desirable location and $F(y_i)$ denotes the price per night) yielding an analogous sufficient condition for discrete nestedness.

%{\color{blue} BP: I don't fully follow this discussion.  What do $x_1$ and $x_2$ represent?}

\subsection{Numerical Algorithm}\label{DiscContinuous}
In this section, we present a numerical algorithm for computing the solution of~\eqref{cong} under the assumption of discrete nestedness. We first present the numerical formulation of the problem introduced earlier. 

\begin{defn}[Laguerre Cells and Tessellation]
    Given a finite set of points $\{y_i\}_{i=1}^N \subset Y$, a weight vector $\mathbf{v} = (v_i)_{i=1}^N \in \mathbb{R}^N$, and a cost function $c : X \times Y \to \mathbb{R}_{\geq 0}$, the \emph{Laguerre cell} associated with $y_i$ is defined as
    \[
    \text{Lag}_i(\mathbf{v}) := \left\{ x \in X \ \middle| \ c(x, y_i) - v_i \leq c(x, y_j) - v_j,\ \forall j \neq i \right\}.
    \]
    The collection of cells $\{\text{Lag}_i(\mathbf{v})\}_{i=1}^N$ constitutes a \emph{Laguerre tessellation} of the domain $X$. Furthermore, the tessellation is said to be \emph{nested} if each Laguerre cell $\text{Lag}_i(\mathbf{v})$ shares a boundary with at most two adjacent cells, namely $\text{Lag}_{i-1}(\mathbf{v})$ and $\text{Lag}_{i+1}(\mathbf{v})$, for all $1 < i < N$.
\end{defn}

\begin{prob}\
    Given a set of points $\{y_i\}_{i=1}^N$, a probability measure $\mu(x)$ on the domain $X$, and a cost function $c$, we seek to solve the following system of equations:
    \begin{gather}\label{Prob1}
        \mu\left(\text{Lag}_i(\mathbf{v})\right) - (f')^{-1}(C - v_i) = 0, \quad i = 1, 2, \dots, N.
    \end{gather}
    In particular, if $f(z) = z \log z$, then $(f')^{-1}(z) = e^{z}$. Substituting this expression into the system~\eqref{Prob1} yields:
    \begin{gather}\label{Prob2}
        \mu\left(\text{Lag}_i(\mathbf{v})\right) - e^{C - v_i} = 0, \quad i = 1, 2, \dots, N.
    \end{gather}
\end{prob}
\begin{rem}
    To satisfy the compatibility condition in~\eqref{Prob2}, the following normalization must hold:
    \begin{gather*}
        \sum_{i=1}^N \mu\left(\text{Lag}_i(\mathbf{v})\right) = 1 \quad \Longrightarrow \quad \sum_{i=1}^N e^{C - v_i} = 1.
    \end{gather*}
    One approach to enforce this condition is to determine the value of $C$ that satisfies the constraint, and then solve the resulting system of equations. Specifically, $C$ is given by:
    \begin{gather*}
        C = \log\left( \frac{1}{\sum_{i=1}^N e^{-v_i}} \right).
    \end{gather*}
    Substituting this expression into the system~\eqref{Prob2} yields the following:
    \begin{gather}\label{Prob3}
        \mu\left(\text{Lag}_i(\mathbf{v})\right) - \frac{e^{-v_i}}{\sum_{k=1}^N e^{-v_k}} = 0, \quad i = 1, 2, \dots, N.
    \end{gather}
\end{rem}

With the problem reformulated in a form amenable to numerical computation, we adopt Newton's method as a baseline for solving the system in~\eqref{Prob3}. Known for its fast local convergence, Newton’s method is a natural first choice for root-finding problems. However, it suffers from stability and global convergence issues, which motivate the exploration of alternative methods. Additionally, each iteration requires solving a dense linear system, leading to a per-iteration cost of $\mathcal{O}(N^3)$ and an overall complexity of up to $\mathcal{O}(N^4)$ in practice, which can be prohibitive for large $N$. 
% {\color{red} OA: is this true? \color{blue} DO: The first part is definitely true, but $\mathcal{O}(N^4)$ is the conservative upper bound, which means one need $\mathcal{O}(N)$ iterations for convergence.}
We include standard Newton’s method here both to clarify its implementation and to provide a benchmark for comparison.

\subsubsection{Standard Newton's Method}
% The problem in \eqref{Prob3} can be solved using Newton's method. 
In our examples, we utilized the Centered Finite Difference scheme to compute the Jacobian matrix, and used the zero vector, $\mathbf{v^{0}} = \mathbf{0}$, as the initial guess.
    
It is important to note that, as observed in \eqref{Prob3}, the vector $\mathbf{v}$ is unique only up to scalar addition. Consequently, the inverse of the Jacobian matrix is computed in the orthogonal complement of its kernel, which is spanned by $\mathbf{1}$.

\begin{prob}[Standard Newton's Method]
    Consider the problem of determining $\mathbf{v}^*$, a vector in $\mathbb{R}^N$, such that $G(\mathbf{v}^*) = \mathbf{0}$ for the function $G(\mathbf{v}): \mathbb{R}^N \to \mathbb{R}^N$. The components of this function are defined by the equation
    \begin{gather}\label{eq_1}
        \{G(\mathbf{v})\}_i = \mu(\text{Lag}_i(\mathbf{v})) - \frac{e^{-v_i}}{\sum_{k=1}^N e^{-v_k}}, \quad i=1,2,\dots, N\ .
    \end{gather}
    In order to find the root of equation \eqref{eq_1}, Newton's method is used, which iteratively updates the estimate according to the formula:
    \begin{gather*}
        \mathbf{v}^{(k+1)} = \mathbf{v}^{(k)} - [\nabla G(\mathbf{v}^{(k)})]^{\dagger} G(\mathbf{v}^{(k)})\ ,
    \end{gather*}
    where $[\nabla G]^{\dagger}$ denotes the inverse of the derivative $\nabla G$ over the orthogonal complement of the kernel of $\nabla G$, specifically $\text{ker}(\nabla G) = \mathbf{1}$. 
\end{prob}

\subsubsection{Damped Newton's Method}

When the solution corresponds to a nested tessellation, we apply a damped Newton method to ensure iterates remain within the nested domain.

\begin{prob}[Damped Newton's Method]
    Starting from the current iterate $\mathbf{v}^{(k)}$, the update is given by
    \[
        \mathbf{v}^{(k+1)} = \mathbf{v}^{(k)} - \frac{1}{2^s} [\nabla G(\mathbf{v}^{(k)})]^{\dagger} G(\mathbf{v}^{(k)}),
    \]
    where \( s \in \mathbb{Z}_{\geq 0} \) is the smallest integer such that the new iterate \( \mathbf{v}^{(k+1)} \) lies inside the nested domain.

    A practical criterion for verifying the \textbf{nestedness} condition is that the triple intersections of successive Laguerre cells are empty:
    \[
        \mathrm{Lag}_{i-1}(\mathbf{v}) \cap \mathrm{Lag}_i(\mathbf{v}) \cap \mathrm{Lag}_{i+1}(\mathbf{v}) = \emptyset, \quad i=2,3,\dots, N-1.
    \]
\end{prob}

\subsubsection{Nested Methods}

Unlike Newton’s method, which seeks to solve the full nonlinear system simultaneously, the nested structure of the problem admits a sequential approach that greatly simplifies computation. This method exploits the monotonicity and ordered behavior of Laguerre cells under appropriate conditions, allowing the potentials to be computed one at a time in a forward pass. By fixing $v_1 = 0$ and constraining the range of $C$ (as discussed in Remark~\ref{c range}), the vector $\mathbf{v}$ can be constructed component-by-component, preserving the nestedness of the Laguerre tessellation at each step.

This approach is particularly effective when the solution is known or expected to lie within the nested regime, as it avoids costly Jacobian inversions and improves numerical stability. The full formulation is provided below, with pseudocode in Algorithms~\ref{Alg_ErrComp}, \ref{Alg_NestBisect}, and~\ref{Alg_NestNewt}.  Convergence guarantees are established in Appendix~\ref{convergence}.

\begin{prob}\label{Nest_prob}
    Let $H(\mathbf{v}, C) : \mathbb{R}^{N+1} \to \mathbb{R}^N$ be defined component-wise by
    \begin{gather}\label{eq_2}
        \{H(\mathbf{v}, C)\}_i = \mu\left(\text{Lag}_i(\mathbf{v})\right) - e^{C - v_i}, \quad i = 1, 2, \ldots, N.
    \end{gather}
    The goal is to find $(\mathbf{v}^*, C^*)$ such that $H(\mathbf{v}^*, C^*) = \mathbf{0}$. Due to the nested structure of the solution, this system can be solved sequentially. Fixing $v_1 = 0$ and choosing $C \in (-\infty, 0]$, we compute each successive $v_i$ by solving:
    \begin{align*}
        \{H(\mathbf{v}, C)\}_1 &= \mu\left(\text{Lag}_1(v_1, v_2)\right) - e^{C - v_1} = 0 \quad \Rightarrow \quad v_2, \\
        \{H(\mathbf{v}, C)\}_2 &= \mu\left(\text{Lag}_2(v_1, v_2, v_3)\right) - e^{C - v_2} = 0 \quad \Rightarrow \quad v_3, \\
        &\vdots \\
        \{H(\mathbf{v}, C)\}_{N-1} &= \mu\left(\text{Lag}_{N-1}(v_{N-2}, v_{N-1}, v_N)\right) - e^{C - v_{N-1}} = 0 \quad \Rightarrow \quad v_N.
    \end{align*}
    The final component defines an error function:
    \[
        \text{Error}(C) := \{H(\mathbf{v}, C)\}_N = \mu\left(\text{Lag}_N(v_{N-1}, v_N)\right) - e^{C - v_N}.
    \]
    The root $C^*$ satisfying $\text{Error}(C^*) = 0$ can be found using bisection or Newton's method. The derivative $\frac{d}{dC} \text{Error}(C)$ may be approximated using a centered finite difference scheme. Each intermediate equation $\{H(\mathbf{v}, C)\}_i = 0$, for $i = 1, \dots, N-1$, can likewise be solved using either bisection or Newton’s method on $v_{i+1}$.
\end{prob}

Before addressing key practical considerations, we first examine potential challenges in implementing the sequential approach — particularly those related to the nestedness of the Laguerre tessellation and its sensitivity to the parameter $C$. The following remarks outline these issues and suggest strategies for resolving them effectively.

\begin{rem}[Insufficiency of Mass]\label{Nest_Issue2}
    Recall the internal problem in the nested method: given a value of $C$, find $\mathbf{v}^*$ such that $v_1^* = 0$ and the following hold:
    \begin{equation*}
    \begin{aligned}
        &\mu(\text{Lag}_1(v_1^*, v_2)) = e^{C - v_1^*} \Rightarrow v_2^* \quad \text{with} \quad \mu(\text{Lag}_2(v_1^*, v_2^*)) > e^{C - v_2^*}, \\
        &\mu(\text{Lag}_2(v_1^*, v_2^*, v_3)) = e^{C - v_2^*} \Rightarrow v_3^* \quad \text{with} \quad \mu(\text{Lag}_3(v_1^*, v_2^*, v_3^*)) > e^{C - v_3^*}, \\
        &\mu(\text{Lag}_3(v_2^*, v_3^*, v_4)) = e^{C - v_3^*} \Rightarrow v_4^* \quad \text{with} \quad \mu(\text{Lag}_4(v_2^*, v_3^*, v_4^*)) > e^{C - v_4^*}, \\
        &\hspace{2em}\vdots \\
        &\mu(\text{Lag}_{N-2}(v_{N-3}^*, v_{N-2}^*, v_{N-1})) = e^{C - v_{N-2}^*} \Rightarrow v_{N-1}^* \\
        &\hspace{4em}\text{with} \quad \mu(\text{Lag}_{N-1}(v_{N-3}^*, v_{N-2}^*, v_{N-1}^*)) > e^{C - v_{N-1}^*}, \\
        &\mu(\text{Lag}_{N-1}(v_{N-2}^*, v_{N-1}^*, v_N)) = e^{C - v_{N-1}^*} \Rightarrow v_N^*,
    \end{aligned}
    \end{equation*}
    defining the residual
    \[
        \text{Error}(C) := \mu(\text{Lag}_N(v_{N-2}^*, v_{N-1}^*, v_N^*)) - e^{C - v_N^*}.
    \]
    
    The strict inequalities ensure sufficient mass remains at each step to satisfy the next constraint. If any such condition fails, the exponential terms dominate and no feasible solution exists for the given $C$. In this case, $C$ must be decreased. (This corresponds to the situation where $h(C) = -\infty$, as discussed in Appendix~\ref{convergence}.)
\end{rem}

Beyond ensuring structural properties such as nestedness, it is essential to understand the behavior of the error function as the parameter $C$ varies. This understanding is not only of theoretical interest but also plays a critical role in initializing and bounding root-finding algorithms, particularly bisection and Newton’s method. The following remark analyzes the limiting behavior of $\text{Error}(C)$ and its implications for numerical stability and algorithmic initialization.

\begin{rem}[Limiting Behavior of $C$]\label{Nest_Limit}
    Let $y_i$ denote six target points distributed along a quarter circle. Figure~\ref{fig:Err_plot} presents three plots illustrating the limiting behavior of the error function $\text{Error}(C)$. Several key observations follow:

    \begin{itemize}
        \item As shown in Figure~\ref{fig:QCircle1}, the slope of the error curve becomes steeper as the number of target points increases. This supports the hypothesis that the initial interval for bisection in solving Problem~\ref{Nest_prob} becomes increasingly sensitive with finer discretization.
    
        \item There exists an upper bound $\bar{C} < 0$ such that $\text{Error}(C)$ is undefined for all $C > \bar{C}$. This phenomenon, discussed in Remark~\ref{Nest_Issue2}, reflects the inability to construct a valid tessellation due to insufficient mass. Notably, $\bar{C}$ decreases as the number of target points increases.
    
        \item Regarding the lower limit, the error function appears to converge to 1 as $C \to -\infty$. This is due to the exponential terms vanishing in that limit, leading to $\mu(\text{Lag}_i) \to 0$ for $i = 1, \dots, N-1$, and thus:
        \[
            \text{Error}(C) = \mu(\text{Lag}_N) - e^{C - v_N} \to \mu(\text{Lag}_N) \to 1 - \sum_{i=1}^{N-1} \mu(\text{Lag}_i) \to 1.
        \]
        This asymptotic behavior is clearly visible in Figure~\ref{fig:QCircle2}.
    
        \item Similar to the upper limit, there exists a lower bound $\underline{C} < 0$ below which the solution ceases to be nested; see Figure~\ref{fig:QCircle3}. However, observing this behavior requires selecting very small values of $C$, which rarely occurs in practice. As a result, the lower bound $\underline{C}$ is typically less restrictive than the upper bound $\bar{C}$.
        % \textcolor{red}{Does the last part of Remark~\ref{Nest_Issue1} clarify this?}
    \end{itemize}
\end{rem}

\begin{figure}[ht]
    \centering
    \begin{subfigure}{.32\textwidth}
        \centering
        \includegraphics[width=\linewidth]{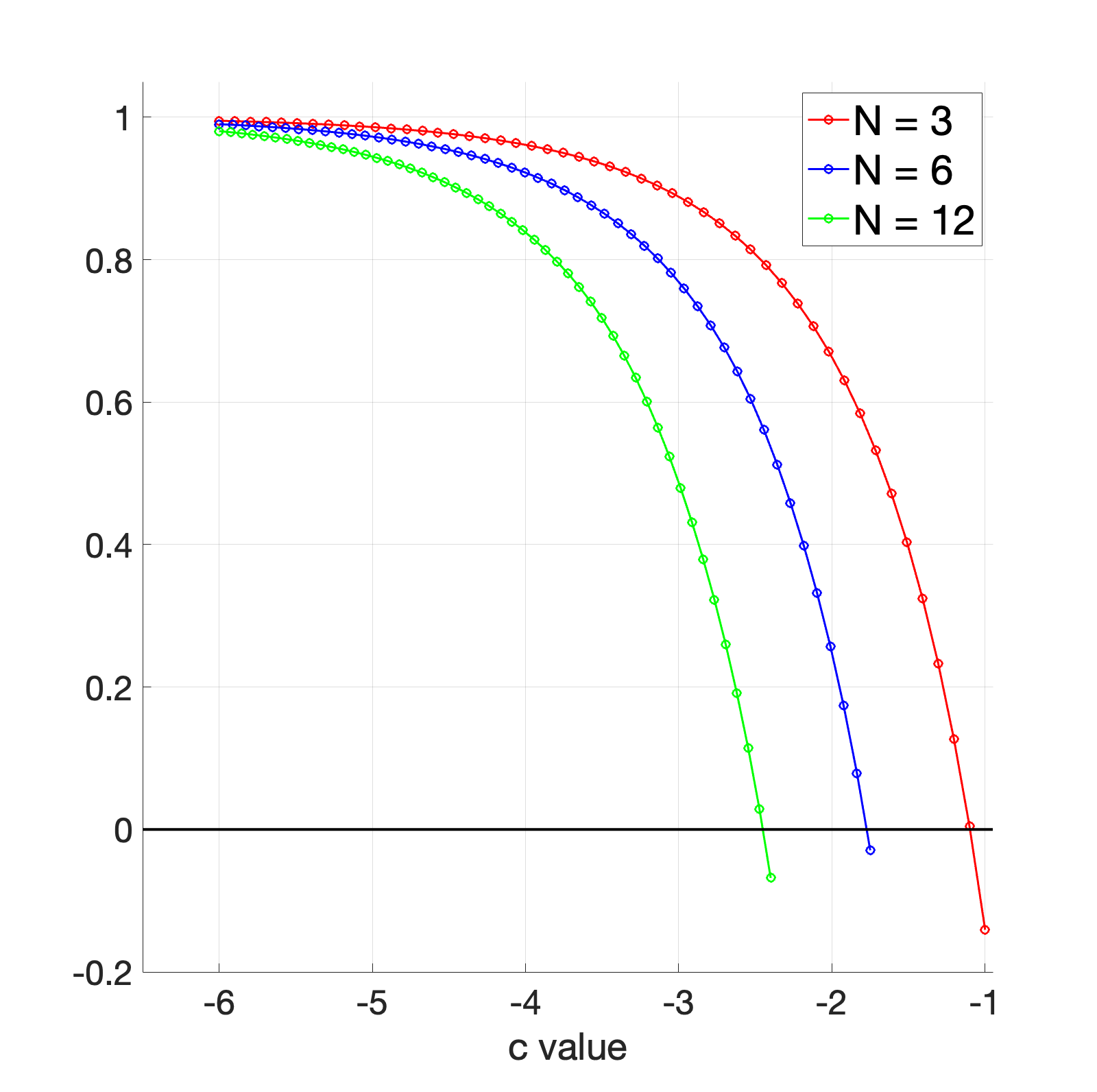}
        \caption{Error as a function of $C$}
        \label{fig:QCircle1}
    \end{subfigure}
    \begin{subfigure}{.32\textwidth}
        \centering
        \includegraphics[width=\linewidth]{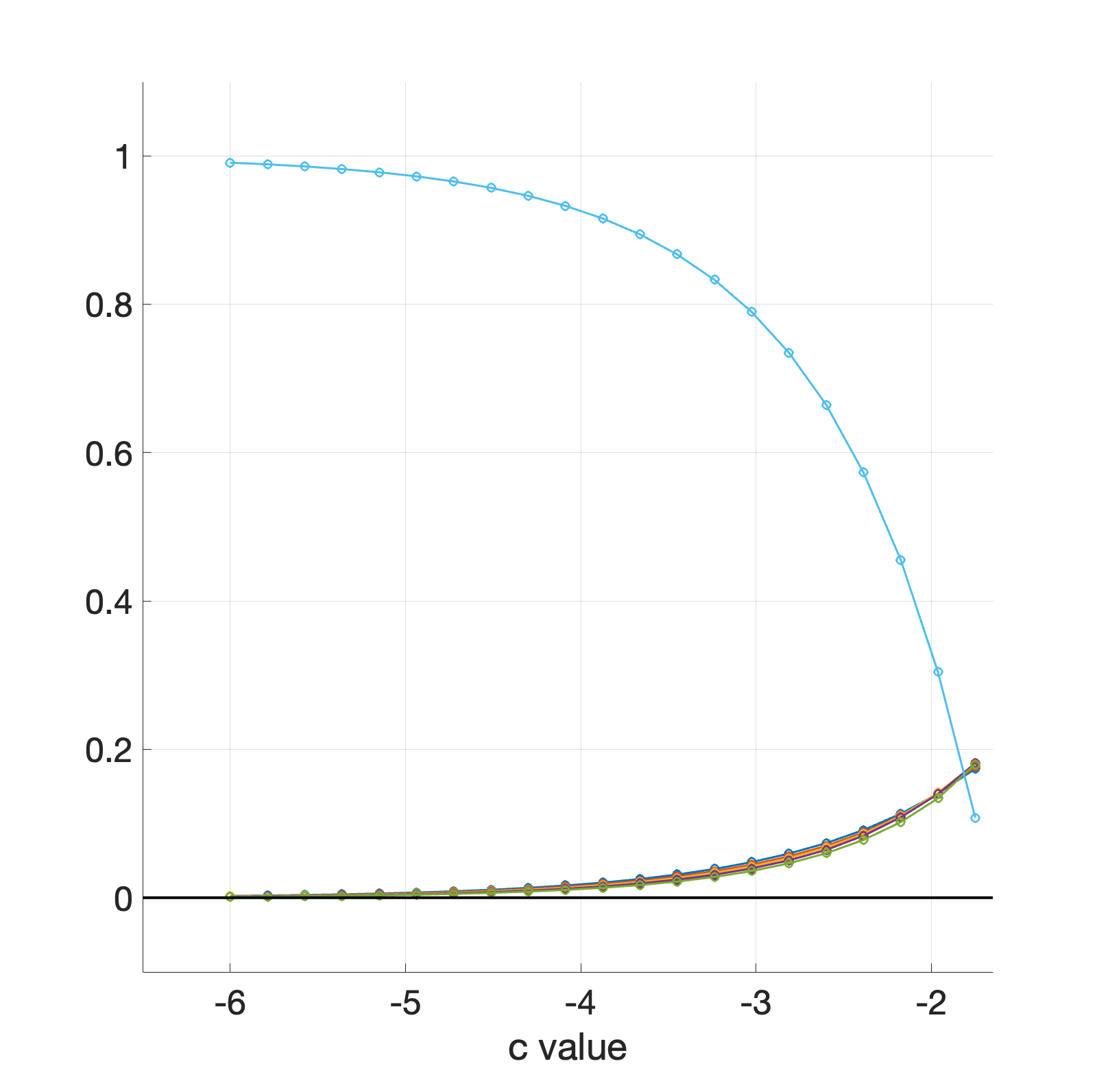}
        \caption{Measures as a function of $C$}
        \label{fig:QCircle2}
    \end{subfigure}
    \begin{subfigure}{.32\textwidth}
        \centering
        \includegraphics[width=\linewidth]{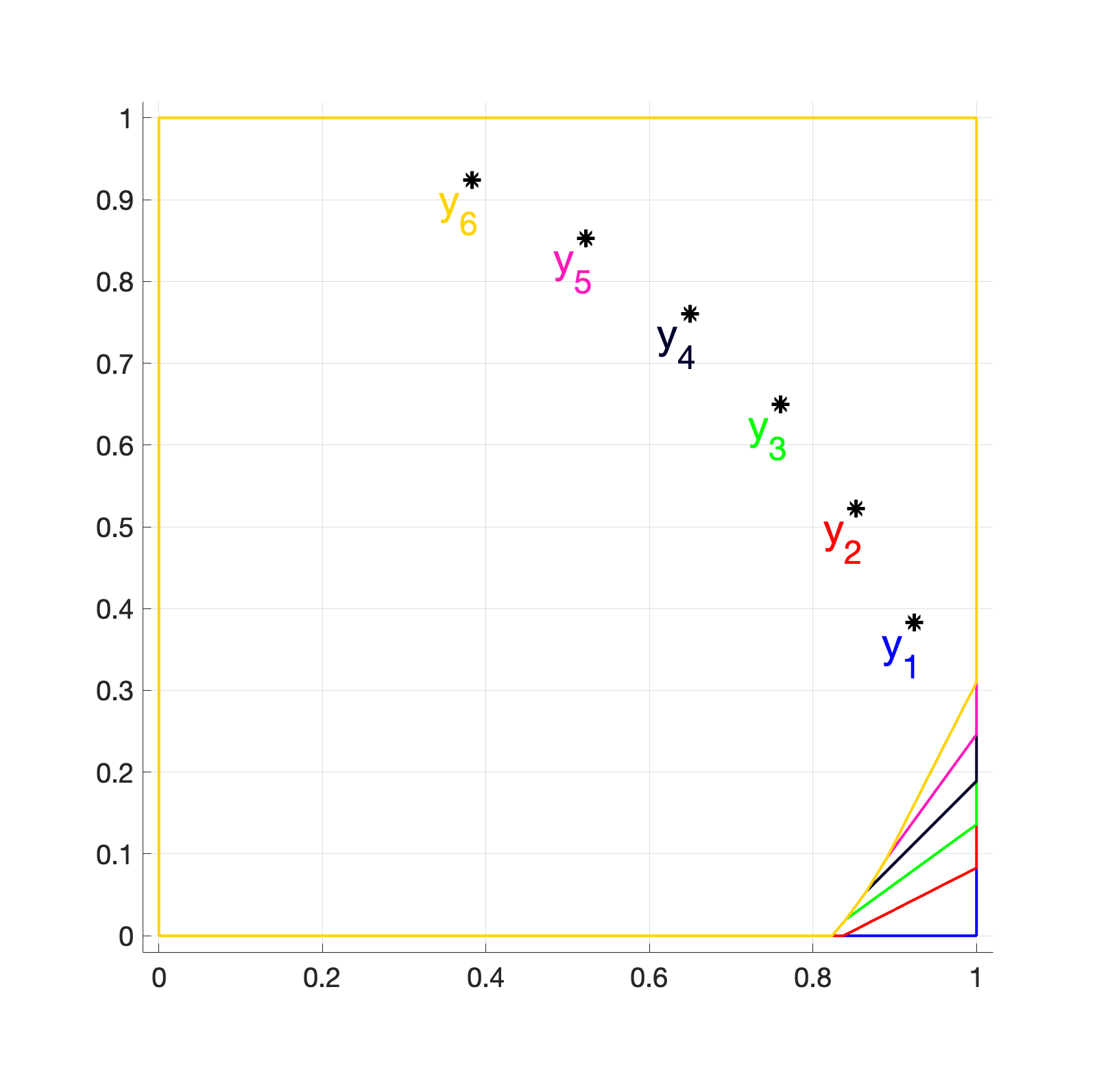}
        \caption{Laguerre cells for small $C$ value}
        \label{fig:QCircle3}
    \end{subfigure}
    \caption{Limiting behavior of the error function}
    \label{fig:Err_plot}
\end{figure}

\subsection{Numerical Examples}\label{numerical_examples}
In this section, we present numerical results that illustrate the performance of the methods described above. We compare the convergence behavior and computational efficiency of standard and damped Newton's method and  the bisection and Newton methods applied to solve the nested problem.  For the initial guesses, we used $\mathbf{v}^{0} = \mathbf{0}$ for both the Newton and Damped Newton methods. For the nested method solved using the Newton algorithm, the initial value was set to $C^0 = -5$ (or $C^0 = -7.5$ when $N = 192$). When employing the bisection method for the nested formulation, the initial interval for $C^0$ was taken as $[-5,\ 0]$ (or $[-7.5,\ 0]$ when $N = 192$).

%% DO: Do we need this paragraph since it is repeated below?
% As shown in the numerical results below, the nested method outperforms Newton's method whenever the solution is nested and the number of target points $N$ is sufficiently large. This behavior is expected: while Newton's method has cubic computational complexity in $N$, the nested approach exploits the problem structure and scales  near-linearly in $N$, 
% % \{\textcolor{red}{OA: is it the case?} \textcolor{blue}{DO: I believe it should be linear or superlinear since the internal equations scale linearly when the nested structure is used}\}
% leading to significant gains in efficiency. This behavior arises because the nested method decomposes the problem into $N$ scalar equations, each of which can be solved with a complexity that does not depend on $N.$

\renewcommand\theequation{E\arabic{equation}} \setcounter{equation}{0}

We consider the following examples: %{\color{blue}BP: Is the cost $-x\cdot y(t)$ here?}
\begin{align}
    &\text{Straight line:} \quad y(t) = \icol{t \\ t}, \quad t \in \left[\frac{1}{10}, \frac{9}{10}\right], \label{E1} \\
    &\text{Scaled parabola:} \quad y(t) = \icol{t \\ \left(\frac{t}{e}\right)^2}, \quad t \in [0, 1], \label{E2}\\
    &\text{Quarter circle:}\quad y(t) = \icol{\cos(t)\\\sin(t)}, \quad t\in \bigg{[}\frac{\pi}{8}, \frac{3\pi}{8}\bigg{]}, \label{E3}\\
    &\text{Parabola:} \quad y(t) = \icol{t\\t^2}, \quad t\in \bigg{[}\frac{1}{N+1}, \frac{N}{N+1}\bigg{]}\label{E4}
\end{align}
\renewcommand\theequation{\arabic{equation}} \setcounter{equation}{9}
where we fix the cost function to be $c(x,y)=\|x-y\|_2^2.$

For a given value of $N$, the target points are placed equidistantly over the respective parameter intervals.

It is well known that the solution to Example~\eqref{E1} is always nested. Moreover, as demonstrated in Example~\ref{example suffic}, the solution to Example~\eqref{E2} with a uniform measure is also nested.
Tables~\ref{tab:Unif} and \ref{tab:NonUnif} report the computation time and number of iterations for each method, along with the number of damping steps applied.

\begin{table}[ht]
    \centering\scalebox{0.6}{\setstretch{2.5}
    \begin{tabular}{||c||c|c|c|c|c|c|c||}\hline\hline
        \textbf{Number of Target Points:} & $\mathbf{N=3}$ & $\mathbf{N=6}$ & $\mathbf{N=12}$ & $\mathbf{N=24}$ & $\mathbf{N=48}$ & $\mathbf{N=96}$ & $\mathbf{N=192}$\\\hline\hline
        
        \multicolumn{8}{c}{\Large \textbf{Example \eqref{E1}}}\vspace{0.25em}\\\cline{1-8}\cline{1-8}
        \textbf{Solution}               & $C = -1.1532$      & $C = -1.8491$      & $C = -2.531$       & $C = -3.2152$ & $C = -3.9028$  & $C = -4.5928$ & $C= -5.2842$ \\\hline
        \textbf{Standard Newton}        & $ 1.2091$ (3)      &  $ 3.4725$ (3)     &  $ 14.082$ (3)     & \cellcolor{green!25} $54.024$ (3) & \cellcolor{green!25} $259.97$ (3) & $1010.6$ (3)& $8235.2$ (3) \\\hline
        \textbf{Damped Newton}          & \cellcolor{green!25} $1.1721$ (3, 0)  &  \cellcolor{green!25} $3.3883$ (3, 0) &  \cellcolor{green!25} $13.796$ (3, 0) & $54.578$ (3, 0) & $260.45$ (3, 0) & \cellcolor{green!25} $1005.1$ (3,0) & \cellcolor{green!25} $8187.2$ (3,0) \\\hline
        \textbf{Nested Bisection}       & $30.803$ (16)      & $180.07$ (18)      & $393.44$ (18)      & $971.6$ (17) & $2168.8$ (16) & $4804.6$ (18) & $10317$ (18) \\\hline
        \textbf{Nested Newton} & $33.221$ (5)       &  $137.06$ (5)      &  $400.42$ (6)      & $887.59$ (5) & $2133$ (5) & $3896.6$ (4) & $9805.5$ (6) \\\hline
        \hline

        \multicolumn{8}{c}{\Large \textbf{Example \eqref{E2}}}\vspace{0.25em}\\\cline{1-8}\cline{1-8}
        \textbf{Solution}               & $C = -1.1609$    & $C = -1.8478$    & $C = -2.5315$   & $C = -3.2188$     & $C = -3.2188$     & $C = -4.6002$ & $C = -5.2925$ \\\hline
        \textbf{Standard Newton}        & \cellcolor{green!25} $0.5628$ (2)     & $1.9745$ (2)     &  $10.525$ (3)   & $48.172$ (3)      & \cellcolor{green!25} $345.57$ (4)      & \cellcolor{red!25} NAN & $18361$ (5) \\\hline
        \textbf{Damped Newton}          & $0.6491$ (2, 0)  & \cellcolor{green!25} $1.9734$ (2, 0)  & \cellcolor{green!25} $9.8136$ (3, 0) & \cellcolor{green!25} $48.102$ (3,0)    & $352.84$ (4, 0)   & $1974.5$ (4, 0) & $18381$ (5, 0) \\\hline
        \textbf{Nested Bisection}       & $27.489$ (18)    & $61.153$ (13)    & $223.65$ (17)   & $593.26$ (18)     & $1142.1$ (16)     & $2706.4$ (17) & \cellcolor{green!25} $6698.3$ (18) \\\hline
        \textbf{Nested Newton}& $25.192$ (5)     & $71.369$ (5)     & $195.74$ (5)    & $497.13$ (5)      & $1349.4$ (6)      & $2607.4$ (5) & \cellcolor{red!25} NAN \\\hline
        \hline

        \multicolumn{8}{c}{\Large \textbf{Example \eqref{E3}}}\vspace{0.25em}\\\cline{1-8}\cline{1-8}
        \textbf{Solution}               & $C = -1.0985$      & $C = -1.7721$      & $C = -2.4504$      & $C = -3.1345$ & $C = -3.8225$ & $C = -4.5128$ & $C = -5.2045$ \\\hline
        \textbf{Standard Newton}        & \cellcolor{green!25} $0.40778$ (1)    &  \cellcolor{green!25} $3.6043$ (3)    &  $13.868$ (3) & $69.582$ (4) & $287.73$ (4) & \cellcolor{red!25} NAN & \cellcolor{red!25} NAN \\\hline
        \textbf{Damped Newton}          & $0.4165$ (1, 0)    &  $3.7155$ (3, 0)   & \cellcolor{green!25} $13.692$ (3, 0)  & \cellcolor{green!25} $68.415$ (4,0) & \cellcolor{green!25} $284.54$ (4, 0) & \cellcolor{red!25} NAN & \cellcolor{red!25} NAN \\\hline
        \textbf{Nested Bisection}       & $31.577$ (17)      & $81.748$ (15)      & $200.37$ (17)      & $499.13$ (18) & $1059$ (17) & \cellcolor{green!25} $1553.4$ (18) & \cellcolor{green!25} $3408$ (17) \\\hline
        \textbf{Nested Newton-Bisection}& $26.571$ (4)       &  $105.78$ (6)      &  $ 158.6$ (4)      & $483.55$ (5) & $967.55$ (5) & \cellcolor{red!25} NAN & $4973.4$ (7) \\\hline
        \hline
    \end{tabular}}
    \caption{Computation time (and iterations) with error tolerance $10^{-5}$ and uniform measure $d\mu = dx_1dx_2$}\label{tab:Unif}
\end{table}

\begin{table}[ht]
    \centering\scalebox{0.6}{\setstretch{2.5}
    \begin{tabular}{||c||c|c|c|c|c|c|c||}\hline\hline
        \textbf{Number of Target Points:} & $\mathbf{N=3}$ & $\mathbf{N=6}$ & $\mathbf{N=12}$ & $\mathbf{N=24}$ & $\mathbf{N=48}$ & $\mathbf{N=96}$ & $\mathbf{N=192}$\\\hline\hline
        
        \multicolumn{7}{c}{\Large \textbf{Example \eqref{E1}}}\vspace{0.25em}\\\cline{1-8}\cline{1-8}
        \textbf{Solution}               & $C = -1.403$      & $C = -2.1136$    & $C = -2.8009$    & $C = -3.4869$     & $C = -4.175$      & $C = -4.8649$ & $C = -5.5562$\\\hline
        \textbf{Standard Newton}        & \cellcolor{red!25} NAN      &  \cellcolor{red!25} NAN     &  \cellcolor{red!25} NAN     & \cellcolor{red!25} NAN & \cellcolor{red!25} NAN & \cellcolor{red!25} NAN & \cellcolor{red!25} NAN \\\hline
        \textbf{Damped Newton}          & \cellcolor{green!25} $0.83118$ (3, 1)  & \cellcolor{green!25} $6.3935$ (5, 2)  & \cellcolor{green!25} $17.78$ (5, 2)  & \cellcolor{green!25} $89.89$ (5, 2)    & \cellcolor{green!25} $290.34$ (4, 2)   & \cellcolor{green!25} $1613.1$ (5, 2) & \cellcolor{red!25} NAN\\\hline
        \textbf{Nested Bisection}       & $36.532$ (18)     & $124.26$ (18)    & $247.53$ (17)   & $386.78$ (12)     & $991.95$ (14)     & $2262.5$ (15) & \cellcolor{green!25} $5128.7$ (16) \\\hline
        \textbf{Nested Newton}& $25.485$ (5)      & $130.64$ (6)     & $202.72$ (4)    & $475.68$ (4)      & $933.64$ (3)      & $2316.2$ (4) & $7468.3$ (6) \\\hline
        \hline

        \multicolumn{7}{c}{\Large \textbf{Example \eqref{E2}}}\vspace{0.25em}\\\cline{1-8}\cline{1-8}
        \textbf{Solution}               & $C = -1.3624$    & $C = -2.0517$      & $C = -2.7334$     & $C =  -3.4182$    & $C = -4.1063$     & $C = -4.7967$ & $C = -5.4884$ \\\hline
        \textbf{Standard Newton}        & \cellcolor{green!25} $0.99633$ (4)    & \cellcolor{red!25} NAN                & \cellcolor{red!25} NAN               & \cellcolor{red!25} NAN               & \cellcolor{red!25} NAN               & \cellcolor{red!25} NAN & \cellcolor{red!25} NAN \\\hline
        \textbf{Damped Newton}          & $1.1657$ (4, 0)  &  \cellcolor{green!25} $3.9864$ (4, 1)   & \cellcolor{green!25} $14.167$ (4, 2)   & \cellcolor{green!25} $72.817$ (5,3)    & \cellcolor{green!25} $371.12$ (6, 5)   & \cellcolor{green!25} $1764.4$ (6, 6) & $13892$ (6, 7) \\\hline
        \textbf{Nested Bisection}       & $28.559$ (17)    & $73.398$ (17)      & $240.78$ (17)     & $445.37$ (17)     & $1009.3$ (16)     & $2278.1$ (17) & \cellcolor{green!25} $4481.4$ (14) \\\hline
        \textbf{Nested Newton}& $31.632$ (6)     & $101.99$ (7)       & $177.07$ (4)      & $281.06$ (3)      & $1103.6$ (5)      & $2067.2$ (4) & $6707.2$ (6) \\\hline
        \hline

        \multicolumn{7}{c}{\Large \textbf{Example \eqref{E4}}}\vspace{0.25em}\\\cline{1-8}\cline{1-8}
        \textbf{Solution} & $C = -1.3593$      & $C = -2.1325$      & $C = -2.8632$    & $C = -3.573$     & $C = -4.2735$ & $C = -4.97$ &  \\\hline
        \textbf{Standard Newton} & \cellcolor{red!25} NAN  & \cellcolor{red!25} NAN  & \cellcolor{red!25} NAN & \cellcolor{red!25} NAN & \cellcolor{red!25} NAN & \cellcolor{red!25} NAN & \cellcolor{red!25} NAN \\\hline
        \textbf{Damped Newton}          & \cellcolor{green!25} $1.0792$ (4, 2) & \cellcolor{red!25} NAN  & \cellcolor{red!25} NAN& \cellcolor{red!25} NAN& \cellcolor{red!25} NAN & \cellcolor{red!25} NAN & \textit{Solution is not nested} \\\hline
        \textbf{Nested Bisection}       & $27.762$ (17) & $78.05$ (16)  & $311.73$ (17) & \cellcolor{green!25} $525.18$ (16) & \cellcolor{green!25} $1269.9$ (18) & $2712.8$ (18) & \textit{Solution is not nested} \\\hline
        \textbf{Nested Newton-Bisection}& $24.132$ (5)  & \cellcolor{green!25} $68.115$ (4)  & \cellcolor{green!25} $243.86$ (4)  & $662.3$ (6)      & $1367.8$ (5)  & \cellcolor{green!25} $1687$ (3) & \textit{Solution is not nested} \\\hline
        \hline
    \end{tabular}}
    \caption{Computation time (and iterations) with error tolerance $10^{-5}$ and nonuniform measure $d\mu = 4x_1x_2dx_1dx_2$}\label{tab:NonUnif}
\end{table}

As shown in the tables, Newton-type methods initially outperform the nested approach in terms of computation time. However, their computational complexity scales as $\mathcal{O}(N^3)$ in the worst case. Consequently, the runtime typically increases by a factor of approximately $4$–$5$ when the number of target points is doubled. In contrast, the nested method exhibits nearly linear complexity, since it involves solving scalar equations sequentially. As a result, its runtime increases more modestly—by a factor of about $2$–$3$—when the number of target points is doubled. Therefore, for sufficiently large $N$, the runtime of both approaches becomes comparable, and eventually, the nested method outperforms Newton’s method.

Another important observation is that Newton's method exhibits increasing sensitivity to the choice of initial guess, particularly in the presence of non-uniform measures, as demonstrated in Table~\ref{tab:NonUnif}. In contrast, the nested method remains significantly more robust with respect to initialization.

Figures~\ref{fig:E1}, \ref{fig:E2}, and \ref{fig:E3} illustrate the solutions for each example in the case of $N = 12$ target points. Notably, while the solution to Example~\eqref{E3} is nested under a uniform measure (see Table~\ref{tab:Unif} and Figure~\ref{fig:E3_1}), it ceases to be nested under a non-uniform measure even with three target points, as shown in Figure~\ref{fig:E3_2}. A similar pattern is observed in Example~\eqref{E4}: the solution remains nested under a uniform measure but becomes non-nested when a non-uniform measure is applied with $N = 192$ target points, as illustrated in Table~\ref{tab:NonUnif}.

%{\color{blue}BP: I think we should discuss including more numerical examples, perhaps including some that illustrate the sensitivity of Newton's method to the initial guess.  The main purpose of this paper is to develop these algorithms, so I think it makes sense to invest some space in showcasing how well they work.}

\begin{figure}[ht]
    \centering
    \begin{subfigure}{.49\textwidth}
        \centering
        \includegraphics[width=\linewidth]{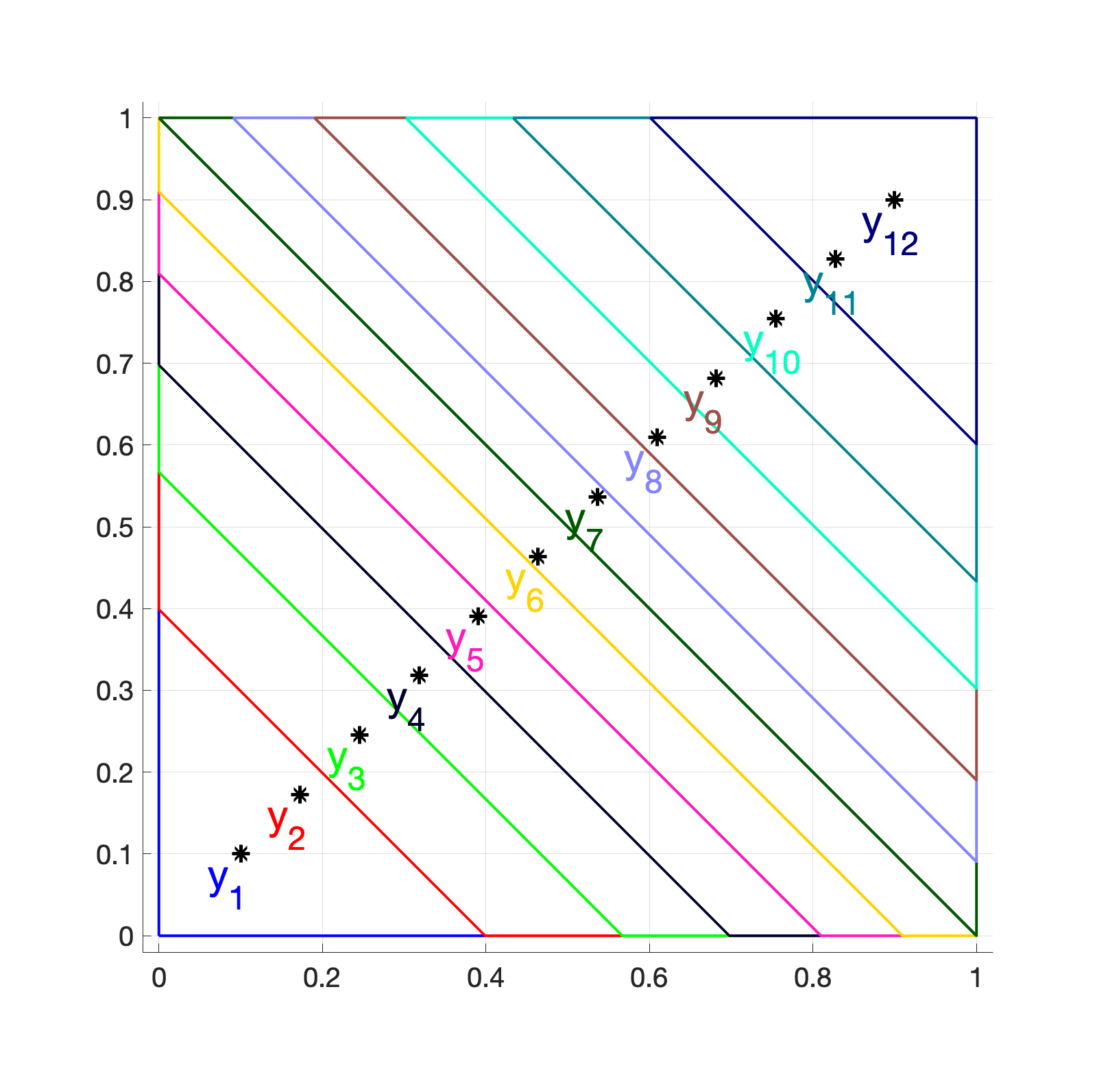}
        \caption{Uniform measure: $d\mu = dx_1dx_2$}
    \end{subfigure}
    \begin{subfigure}{.49\textwidth}
        \centering
        \includegraphics[width=\linewidth]{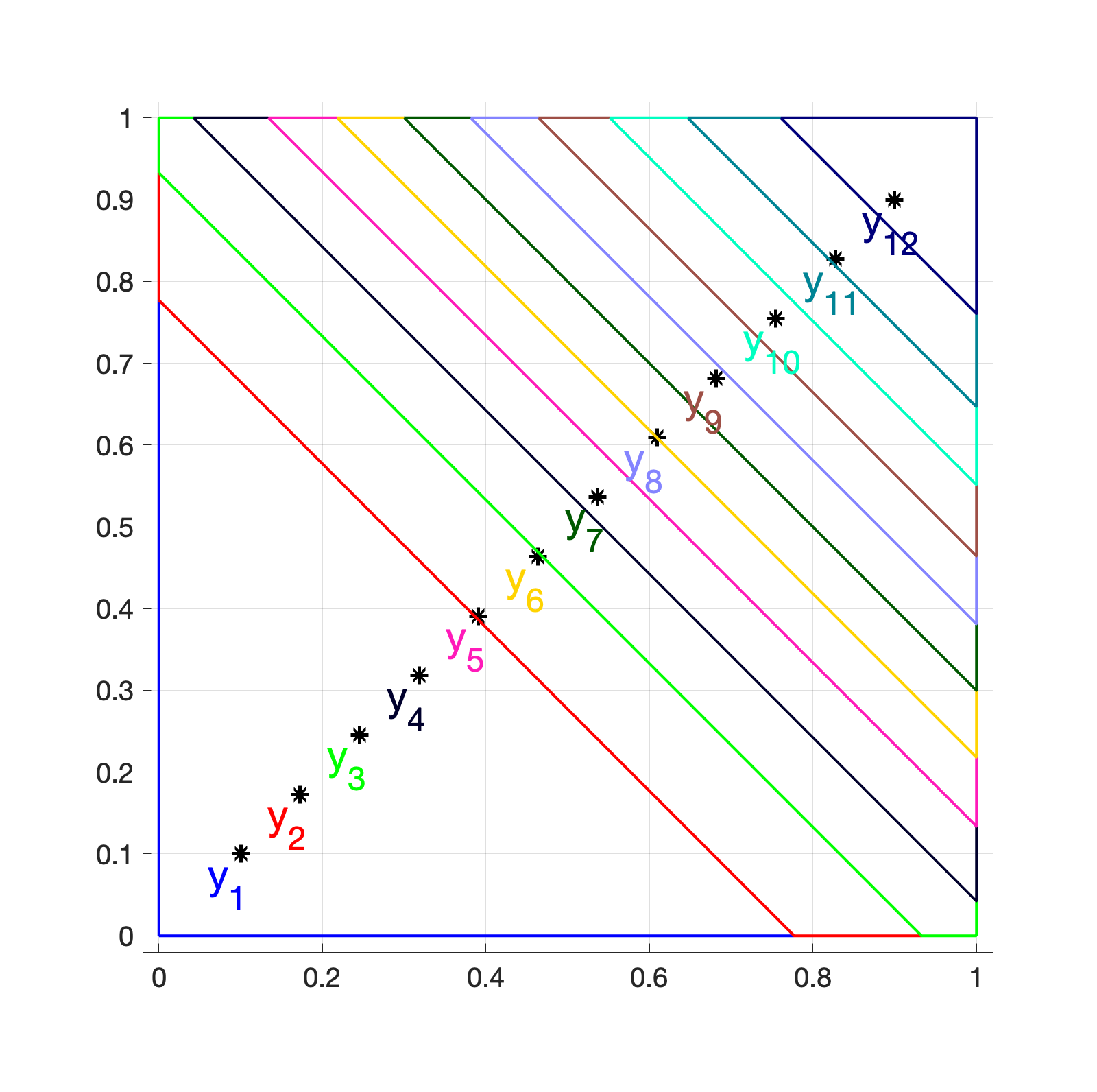}
        \caption{Non-uniform measure: $d\mu = 4x_1x_2dx_1dx_2$}
    \end{subfigure}
    \caption{Solution for Example \eqref{E1} with 12 target points}\label{fig:E1}
\end{figure}

\begin{figure}[ht]
    \centering
    \begin{subfigure}{.49\textwidth}
        \centering
        \includegraphics[width=\linewidth]{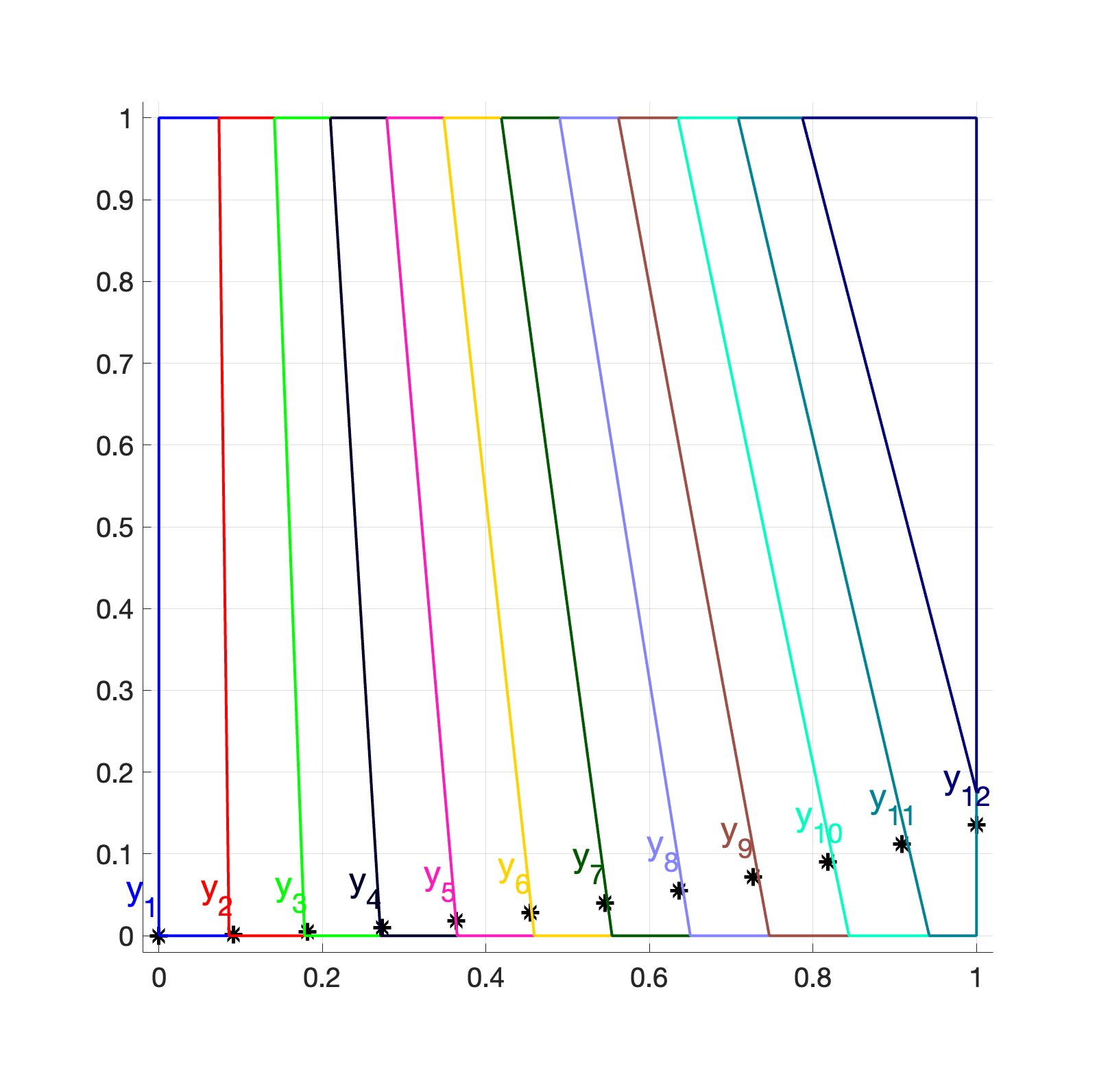}
        \caption{Uniform measure: $d\mu = dx_1dx_2$}
    \end{subfigure}
    \begin{subfigure}{.49\textwidth}
        \centering
        \includegraphics[width=\linewidth]{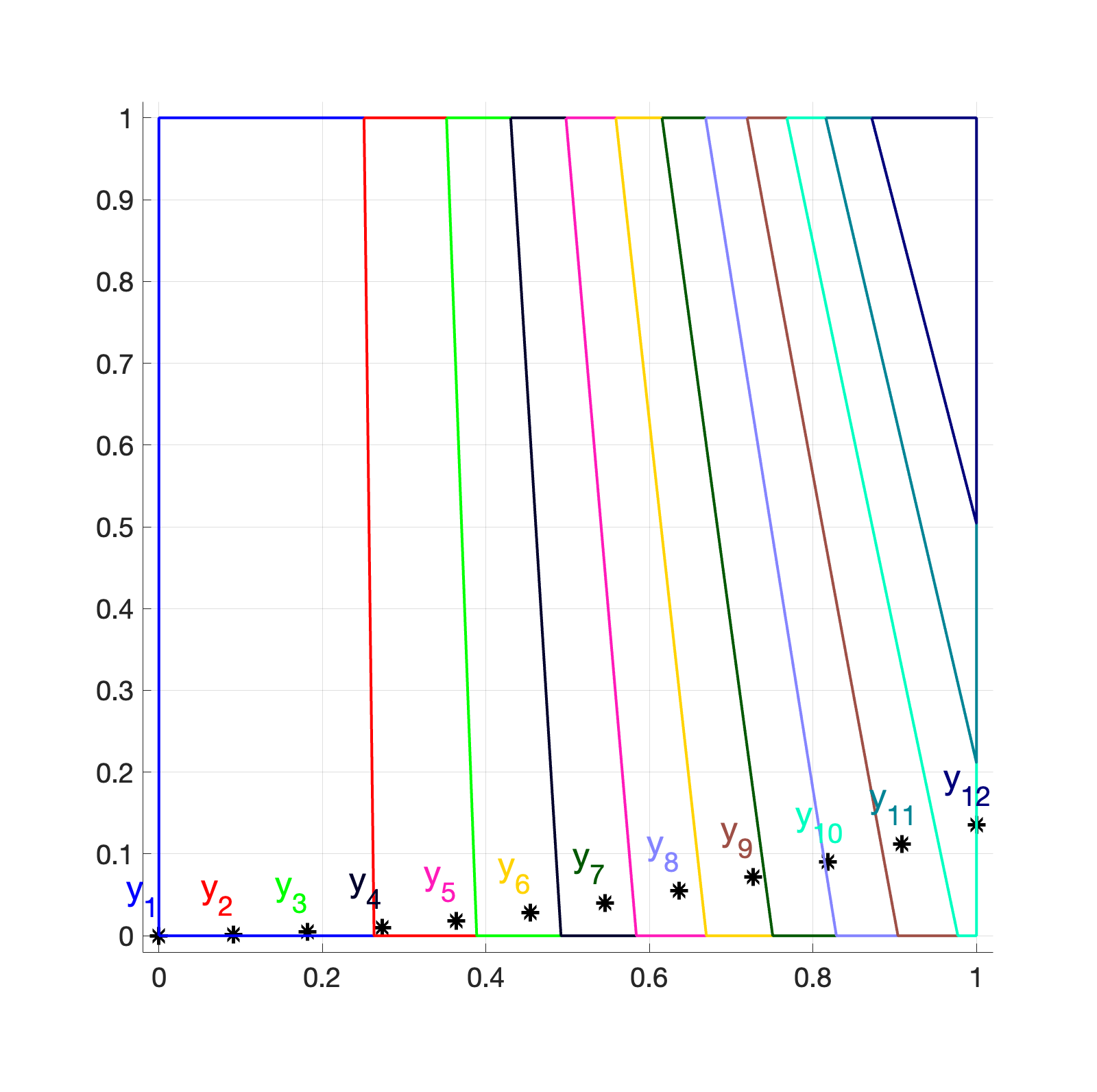}
        \caption{Non-uniform measure: $d\mu = 4x_1x_2dx_1dx_2$}
    \end{subfigure}
    \caption{Solution for Example \eqref{E2} with 12 target points}\label{fig:E2}
\end{figure}

\begin{figure}[ht]
    \centering
    \begin{subfigure}{.49\textwidth}
        \centering
        \includegraphics[width=\linewidth]{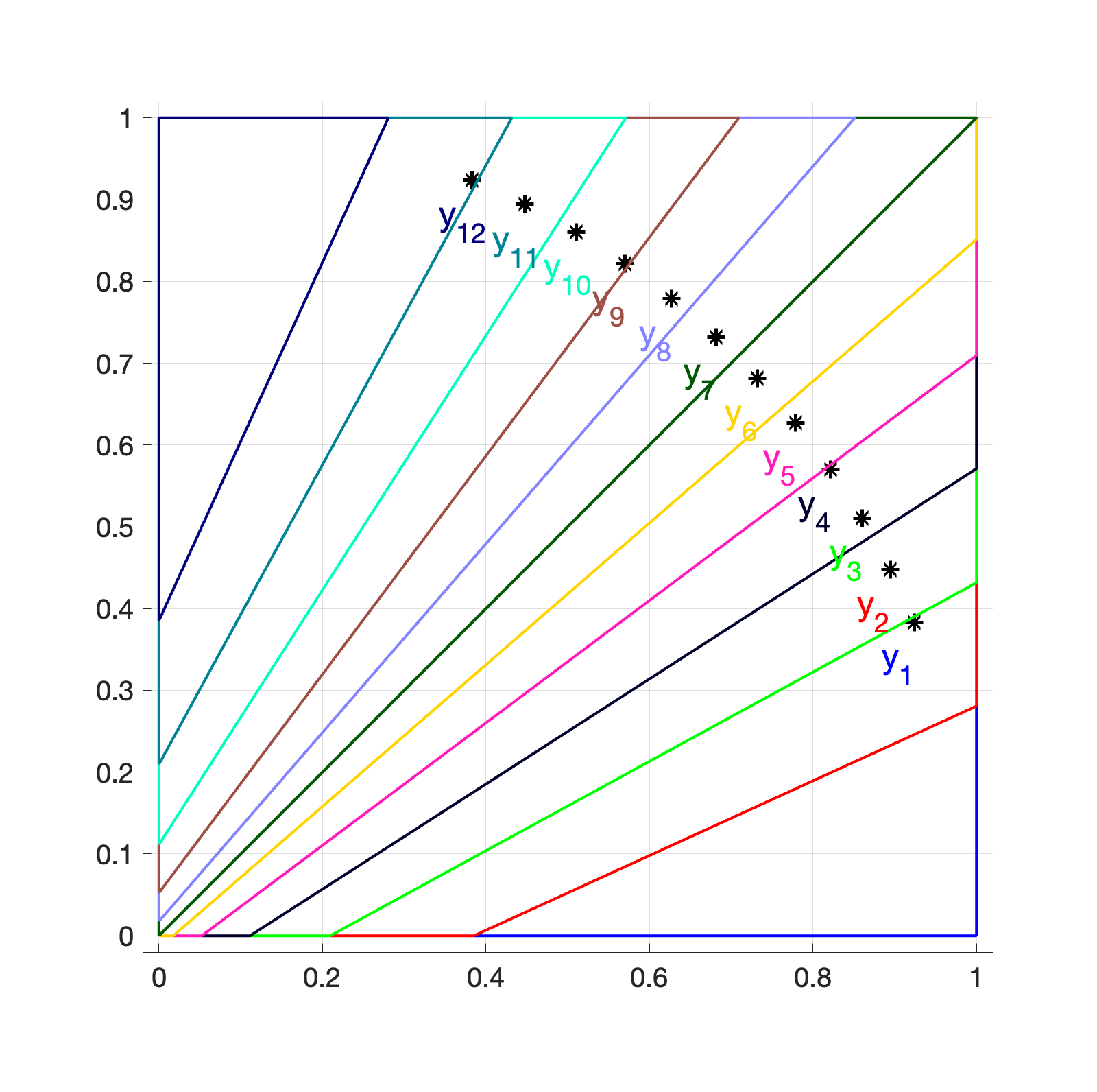}
        \caption{Uniform measure: $d\mu = dx_1dx_2$, $N = 12$}
        \label{fig:E3_1}
    \end{subfigure}
    \begin{subfigure}{.49\textwidth}
        \centering
        \includegraphics[width=\linewidth]{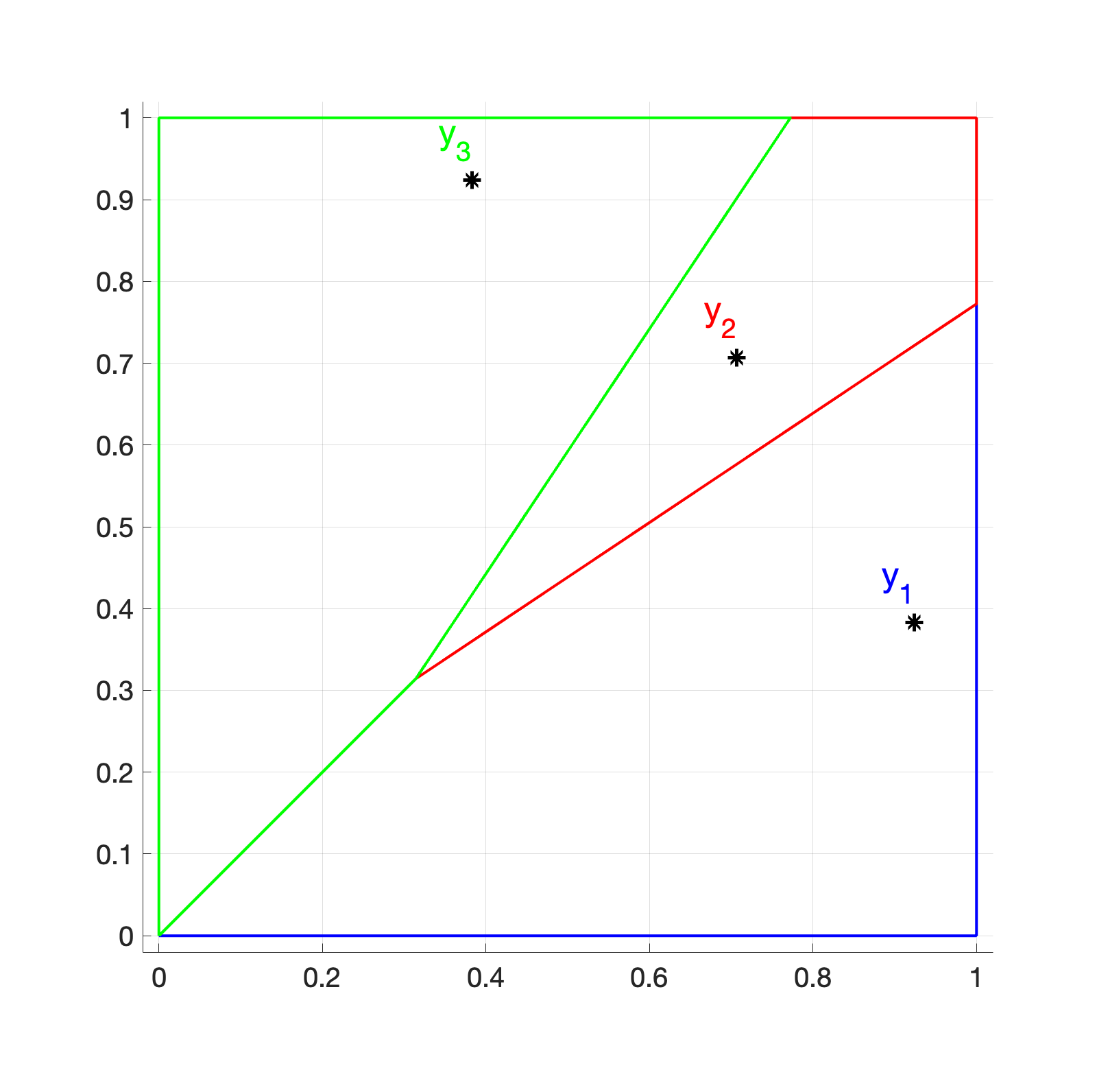}
        \caption{Non-uniform measure: $d\mu = 4x_1x_2dx_1dx_2$, $N = 3$}
        \label{fig:E3_2}
    \end{subfigure}
    \caption{Solution for Example \eqref{E3} with 12 target points}\label{fig:E3}
\end{figure}

\section{Hedonic Pricing Problem}\label{hedonic problem}
We now turn to a class of problems that arise naturally in two-sided matching and market design, often referred to as hedonic pricing problems \cite{COCV_2005__11_1_57_0}. In these settings, agents from two distinct populations are matched indirectly through a shared distribution of outcomes. This leads to a variational formulation where the objective is to find a probability measure $\nu$ on a finite set $Y = \{y_1, \dots, y_N\}$ that balances transport costs from both populations.

This corresponds to minimizing the sum of two optimal transport costs:
\begin{equation}\label{hp}
\min_{\nu \in \mathcal{P}(Y)} \left\{ \mathcal{W}_{c_1}(\mu_1, \nu) + \mathcal{W}_{c_2}(\mu_2, \nu) \right\},
\end{equation}
where $\mu_1$ and $\mu_2$ are continuous distributions representing the two agent populations, and $c_1$, $c_2$ are their respective cost functions.

In \cite{nenna2020variational}, the notion of nestedness was extended to the two-population setting. In this work, we further extend it to the semi-discrete case. Specifically, we define a solution to be \emph{hedonically nested} when both optimal transport plans from the respective populations to the common outcome distribution exhibit discrete nestedness.

\begin{defn}
    The solution $\nu$ of problem \eqref{hp} is \emph{hedonically nested} if both $(c_1,\mu_1,\nu)$ and $(c_2,\mu_2,\nu)$ are discretely nested.
\end{defn}

\subsection{Problem Statement}

Throughout this section, we restrict ourselves to the setting where $X = X^1 = X^2 \subset \mathbb{R}^d$ and $c_1 = c_2 \in C^2(X \times Y)$ where $c$ satisfies the generalized Spence-Mirlees, or twist, condition (found in, for example, \cite{santambrogio2015optimal}). In addition, we are given two absolutely continuous probability measures $\mu_1$ and $\mu_2$ supported on the bounded, open domain $X$, and a collection of target points $y_i \in Y \subset \mathbb{R}^d$, $i = 1, 2, \dots, N$.
From \cite{chiappori2010hedonic}, it is known that the solution to the corresponding hedonic problem is unique, %{\color{blue} BP: I don't think we have uniqueness without some extra conditions on $c$ (such as twistedness).} 
and the optimality condition
\[
    v^1 + v^2 = C
\]
holds for some constant $C$, where $v^1$ and $v^2$ are the dual Kantorovich potentials corresponding to the optimal transport problems $(c, \mu_1, \nu)$ and $(c, \mu_2, \nu)$, respectively.

This motivates the following reformulation:
\begin{prob}\label{prob:Hed2}
Given target points $y_i \in Y \subset \mathbb{R}^d$, $i = 1, \dots, N$, a domain $X \subset \mathbb{R}^d$, and a cost function $c$, find $(\mathbf{v}, C) \in \mathbb{R}^N \times \mathbb{R}$ such that $v_1 = 0$ and
\[
    \mu_1(\text{Lag}_i(\mathbf{v})) = \mu_2(\text{Lag}_i(C - \mathbf{v})), \quad \text{for } i = 1, \dots, N,
\]
where $\mu_1$ and $\mu_2$ are probability measures on $X$, and
\[
    \text{Lag}_i(\mathbf{v}) = \left\{ x \in X \ \middle| \ c(x, y_i) - v_i \leq c(x, y_j) - v_j,\ \forall j \neq i \right\}.
\]
\end{prob}

This reformulation is justified by observing that for the equilibrium measure $\nu$, optimal transport yields $\mu_1(\text{Lag}_i(v^1)) = \mu_2(\text{Lag}_i(v^2)) = \nu_i$ for each $i$, where $v^1$ and $v^2$ are the corresponding dual potentials. Using the optimality condition $v^1 + v^2 = C$, we can set $v^1 = \mathbf{v}=(v_i)_{i=1}^N$ and $v^2 = (C - v_i)_{i=1}^N$, thereby reducing the problem to a single potential vector $\mathbf{v}$ and a constant $C$.

The formulation in Problem~\ref{prob:Hed2} provides a natural framework for analyzing the structure of the hedonic pricing problem. Moreover, in the special case where the solution is hedonically nested, further simplifications are possible. Specifically, the nested structure of the optimal partitions allows for a sequential solution method, which underpins the nested formulation described in the next section.

\begin{comment}
\begin{prob}\label{prob:Hed1}
    Given the target points $y_i\in Y\subset\mathbb{R}^d$, $i=1,2,\dots,N$, the domain $X\subset \mathbb{R}^d$, and the cost function $c(x,y) = ||x-y||_2^2$, find $\mathbf{v}\in\mathbb{R}^N$ such that $v_1 = 0$ and
    \begin{gather*}
        \mu_1(X_i(\mathbf{v})) = \mu_2(X_i(C-\mathbf{v}))\ ,\ i=1,2,\dots,N\ , 
    \end{gather*}
    where $\mu_1\ ,\ \mu_2$ are probability measures on $X$ and
    \begin{gather*}
        X_i(\mathbf{v}) = \{\mathbf{x}\in X\ |\ ||x-y_i||_2^2 - v_i \leq  ||x-y_j||_2^2 - v_j\ ,\ \forall j\neq i\}\ .
    \end{gather*}
\end{prob}
\end{comment}

\textbf{Nested Algorithm for the Hedonic Problem.}  
Assume that the solution to Problem~\ref{prob:Hed2}, denoted by $(\mathbf{v}^*, C^*)$, is \emph{hedonically nested}; that is, both Laguerre tessellations $\{\text{Lag}_i(\mathbf{v}^*)\}_{i=1}^N$ and $\{\text{Lag}_i(C^* - \mathbf{v}^*)\}_{i=1}^N$ are nested. Under this assumption, the hedonic pricing problem admits a sequential formulation:
\begin{gather*}
    v_1 = 0 \ \Rightarrow\ 
    \mu_1(\text{Lag}_1(v_1, v_2)) = \mu_2(\text{Lag}_1(C - v_1, C - v_2)) \ \Rightarrow\ v_2, \\
    \mu_1(\text{Lag}_2(v_1, v_2, v_3)) = \mu_2(\text{Lag}_2(C - v_1, C - v_2, C - v_3)) \ \Rightarrow\ v_3, \\
    \vdots \\
    % \mu_1(\text{Lag}_i(v_{i-1}, v_i, v_{i+1})) = \mu_2(\text{Lag}_i(C - v_{i-1}, C - v_i, C - v_{i+1})) \ \Rightarrow\ v_{i+1}, \\
    % \vdots \\
    \mu_1(\text{Lag}_{N-1}(v_{N-2}, v_{N-1}, v_N)) = \mu_2(\text{Lag}_{N-1}(C - v_{N-2}, C - v_{N-1}, C - v_N)) \ \Rightarrow\ v_N, \\
    \text{Error}(C, \mathbf{v}) := \mu_1(\text{Lag}_N(v_{N-1}, v_N)) - \mu_2(\text{Lag}_N(C - v_{N-1}, C - v_N)).
\end{gather*}

To justify this algorithm, we will show that if the first $N-1$ equations are satisfied, then the final error term $\text{Error}(C, \mathbf{v})$ necessarily vanishes.

\begin{prop}\label{rank}
    Assume that the solution to Problem~\ref{prob:Hed2}, denoted by $(\mathbf{v}^*, C^*)$, is \emph{hedonically nested}. Further, suppose that $(\mathbf{v}^*, C^*)$ satisfies the first $N-1$ equations:
    \begin{gather*}
         \mu_1(\text{Lag}_1(v_1^*, v_2^*)) = \mu_2(\text{Lag}_1(C^* - v_1^*, C^* - v_2^*)), \\
         \mu_1(\text{Lag}_2(v_1^*, v_2^*, v_3^*)) = \mu_2(\text{Lag}_2(C^* - v_1^*, C^* - v_2^*, C^* - v_3^*)), \\
         \vdots \\
         \mu_1(\text{Lag}_{N-1}(v_{N-2}^*, v_{N-1}^*, v_N^*)) = \mu_2(\text{Lag}_{N-1}(C^* - v_{N-2}^*, C^* - v_{N-1}^*, C^* - v_N^*)).
    \end{gather*}
    Then the error function is zero,
    \[
        \text{Error}(C^*, \mathbf{v}^*) := \mu_1(\text{Lag}_N(v_{N-1}^*, v_N^*)) - \mu_2(\text{Lag}_N(C^* - v_{N-1}^*, C^* - v_N^*)) = 0.
    \]
\end{prop}

\begin{proof}
    Since the Laguerre tessellations $\{\text{Lag}_i(\mathbf{v}^*)\}_{i=1}^N$ and $\{\text{Lag}_i(C^* - \mathbf{v}^*)\}_{i=1}^N$ form partitions of $X$, we have:
    \[
        \sum_{i=1}^{N} \mu_1(\text{Lag}_i(\mathbf{v}^*)) = \sum_{i=1}^{N} \mu_2(\text{Lag}_i(C^* - \mathbf{v}^*)) = 1.
    \]
    As a result, we obtain:
    \begin{align*}
        \text{Error}(C^*, \mathbf{v}^*) &= \mu_1(\text{Lag}_N(v_{N-1}^*, v_N^*)) - \mu_2(\text{Lag}_N(C^* - v_{N-1}^*, C^* - v_N^*)) \\
        &= \left(1 - \sum_{i=1}^{N-1} \mu_1(\text{Lag}_i(\mathbf{v}^*))\right) - \left(1 - \sum_{i=1}^{N-1} \mu_2(\text{Lag}_i(C^* - \mathbf{v}^*))\right) \\
        &= \sum_{i=1}^{N-1} \left[\mu_2(\text{Lag}_i(C^* - \mathbf{v}^*)) - \mu_1(\text{Lag}_i(\mathbf{v}^*))\right].
    \end{align*}
    By assumption, the terms in the sum cancel individually, yielding $\text{Error}(C^*, \mathbf{v}^*) = 0$.
\end{proof}

\begin{comment}
\begin{rem}
    Based on the above claim, we deduce that the rank of the hedonic pricing problem is $N-1$ with $N+1$ variables. Consequently, when the solution is hedonically nested, one can freely choose $C$ and $v_1$ and then solve the $N-1$ scalar equations sequentially to obtain the solution. In our experiments, we set $C = 0$ and $v_1 = 0$.
\end{rem}
\end{comment}

\begin{rem}
    By Proposition~\ref{rank}, the hedonic pricing problem has rank $N-1$ with $N+1$ variables. As a result, when the solution is hedonically nested, one can freely specify two degrees of freedom—typically $C$ and $v_1$—and then sequentially solve the remaining $N-1$ scalar equations to recover the full solution. In our numerical experiments, we fix $C = 0$ and $v_1 = 0$. Once the remaining components of $\mathbf{v}$ are determined through this sequential procedure, the uniqueness of the solution ensures that the induced measure $\nu$ is the unique minimizer of the hedonic pricing problem.
\end{rem}

\subsection{Numerical Results}
We present numerical experiments comparing Newton’s method and the nested method for solving the hedonic pricing problem. In the following examples, we focus on the case $c(x,y) = \|x - y\|_2^2$,  where $x \in X \subseteq \mathbb{R}^2$ and the discrete variable $y$ lies along a curve $y(t) \in \mathbb{R}^2$, and compare the convergence behavior and computational efficiency of both methods under different update strategies, including damped versus full-step Newton and bisection versus Newton-style updates within the nested algorithm.

\begin{rem}
    Proposition~\ref{rank} implies that the nested algorithm reduces the hedonic pricing problem to solving $N-1$ scalar equations, each depending only on the previously computed value. Once an earlier potential is determined, the subsequent one can be computed sequentially. Within this framework, two solution strategies can be employed:
    \begin{itemize}
        \item \emph{Nested Bisection:} Each scalar equation is solved via the bisection method. This approach is robust and less sensitive to initial guesses, making it suitable when derivative information is unavailable or unreliable. However, it typically converges more slowly than Newton-based methods.
    
        \item \emph{Nested Newton:} Each scalar equation is solved using Newton’s method, possibly combined with bisection for initialization or as a safeguard. Although Newton’s method requires a reasonable initial guess, such guesses are often naturally provided by the sequential structure, enabling fast and accurate convergence when derivative information is accessible.
    \end{itemize}
\end{rem}

Table~\ref{hedonic table} compares the performance of the nested strategies—\emph{Nested Bisection} and \emph{Nested Newton}—against global (vector-based) Newton methods (both standard and damped) across a range of matching scenarios.

\begin{table}[ht]
    \centering\scalebox{0.6}{\setstretch{2.5}
    \begin{tabular}{||c||c|c|c|c|c|c|c||}\hline\hline
        $\mathbf{v}^{0} = \mathbf{0}\ ;\ v_1 = 0\ ,\ C = 0$ & $\mathbf{N=3}$ & $\mathbf{N=6}$ & $\mathbf{N=12}$ & $\mathbf{N=24}$ & $\mathbf{N=48}$ & $\mathbf{N=96}$ & $\mathbf{N=192}$ \\\hline\hline
        
        \multicolumn{8}{c}{\Large \textbf{Straight line:} $y(t) = \icol{t\\t}$, $t\in [\frac{1}{10}, \frac{9}{10}]$}\vspace{0.25em}\\\cline{1-8}\cline{1-8}
        \textbf{Standard Newton}        & $1.9975$ (3)     & $6.6196$ (3)     & \cellcolor{green!25} $34.171$ (4)     & $147.98$ (4)     & $809.25$ (4)     & $2988.3$ (4)     & \cellcolor{red!25} NAN\\\hline
        \textbf{Damped Newton}          & \cellcolor{green!25} $1.9513$ (3, 0)& \cellcolor{green!25} $6.5805$ (3, 0)  & $34.838$ (4, 0)  & $153.1$ (4, 0)   & $810.58$ (4, 0)  & $2963.4$ (4, 0)  & \cellcolor{red!25} NAN\\\hline
        \textbf{Nested Bisection}       & $4.7031$                            & $16.535$                              & $34.841$         & \cellcolor{green!25} $83.676$         & \cellcolor{green!25} $184.5$          & \cellcolor{green!25} $432.54$ & $747.66$ \\\hline
        \textbf{Nested Newton}& $2.3825$         & $7.5537$         & \cellcolor{red!25} NAN              & \cellcolor{red!25} NAN              & \cellcolor{red!25} NAN              & \cellcolor{red!25} NAN & \cellcolor{green!25} $401.78$ \\\hline
        \hline

        \multicolumn{8}{c}{\Large \textbf{Curve:} $y(t) = \icol{t\\t^{1.5}}$, $t\in [\frac{1}{N+1}, \frac{N}{N+1}]$}\vspace{0.25em}\\\cline{1-8}\cline{1-8}
        \textbf{Standard Newton}        & $1.9493$ (3)     & $15.658$ (4)     & $36.181$ (4)     & $230.68$ (5)     & $861$ (5)       & $4118.6$ (5)     & \cellcolor{red!25} NAN\\\hline
        \textbf{Damped Newton}          & $1.9464$ (3, 0)  & $15.37$ (4, 0)   & $36.177$ (4, 0)  & $230.94$ (5, 0)  & $860.11$ (5, 0) & $4118.6$ (5, 0)  & \cellcolor{red!25} NAN\\\hline
        \textbf{Nested Bisection}       & $4.1779$         & $15.54$          & $29.652$         & \cellcolor{green!25} $73.273$      & \cellcolor{green!25} $167.08$        & \cellcolor{green!25} $380.2$ & \cellcolor{green!25} $671.5$     \\\hline
        \textbf{Nested Newton}& \cellcolor{green!25} $1.6778$         & \cellcolor{green!25} $7.1348$         & \cellcolor{green!25} $14.598$         & \cellcolor{red!25} NAN              & \cellcolor{red!25} NAN             & \cellcolor{red!25} NAN & \cellcolor{red!25} NAN \\\hline
        \hline

        \multicolumn{8}{c}{\Large \textbf{Scaled Parabola:} $y(t) = \icol{t\\ \big{(}\frac{t}{e}\big{)}^2}$, $t\in [0, 1]$}\vspace{0.25em}\\\cline{1-8}\cline{1-8}
        \textbf{Standard Newton}        & $2.1782$ (3)     & $5.6952$ (3)     & $21.281$ (3)     & $95.268$ (3)     & $553.08$ (4)     & \cellcolor{red!25} NAN & \cellcolor{red!25} NAN\\\hline
        \textbf{Damped Newton}          & $1.8582$ (3, 0)  & $5.9949$ (3, 0)  & $21.188$ (3, 0)  & $96.860$ (3, 0)  & $552.20$ (4, 0)  & Solution is not \textit{hedonically} nested & \cellcolor{red!25} NAN\\\hline
        \textbf{Nested Bisection}       & $3.4519$         & $9.7463$         & $22.814$         & $61.754$         & $132.22$         & Solution is not \textit{hedonically} nested & $573.88$\\\hline
        \textbf{Nested Newton}& \cellcolor{green!25} $1.5977$         & \cellcolor{green!25} $4.1966$         & \cellcolor{green!25} $12.116$         & \cellcolor{green!25} $34.122$         & \cellcolor{green!25} $72.963$         & Solution is not \textit{hedonically} nested & \cellcolor{green!25} $288.33$\\\hline
        \hline
        
    \end{tabular}}
    \caption{Computation time of different methods (error tolerance is  $10^{-7}$, $d\mu_1(x) = dx_1dx_2$, $d\mu_2(x) = 4x_1x_2dx_1dx_2$)}
    \label{hedonic table}
\end{table}

Figures~\ref{straight line}--\ref{scaled parabola} illustrate cases where the solution is hedonically nested, as the region boundaries under both cases $(c,\mu_1, \nu)$ and $(c,\mu_2, \nu)$ remain non-intersecting within the domain $X$. In contrast, for target points distributed along the scaled parabola $x_2 = \left(\frac{x_1}{e}\right)^2$, the solution remains nested for some values of $N$, but fails to do so for $N = 96$. Hence, by definition, the corresponding solution $\nu$ is not hedonically nested in that case.
%{\color{blue} BP: Shouldn't the Standard and Damped Newton Algorithms still give a value in this case?}

Overall, when hedonic nestedness holds, both the \emph{Nested Bisection} and \emph{Nested Newton} methods significantly outperform global Newton-based approaches, largely due to their sequential, non-iterative structure. While \emph{Nested Bisection} offers enhanced robustness and stability, it is generally slower on average compared to the more efficient \emph{Nested Newton} method.

\begin{figure}[ht]
    \centering
    \begin{subfigure}{.49\textwidth}
        \centering
        \includegraphics[width=\linewidth]{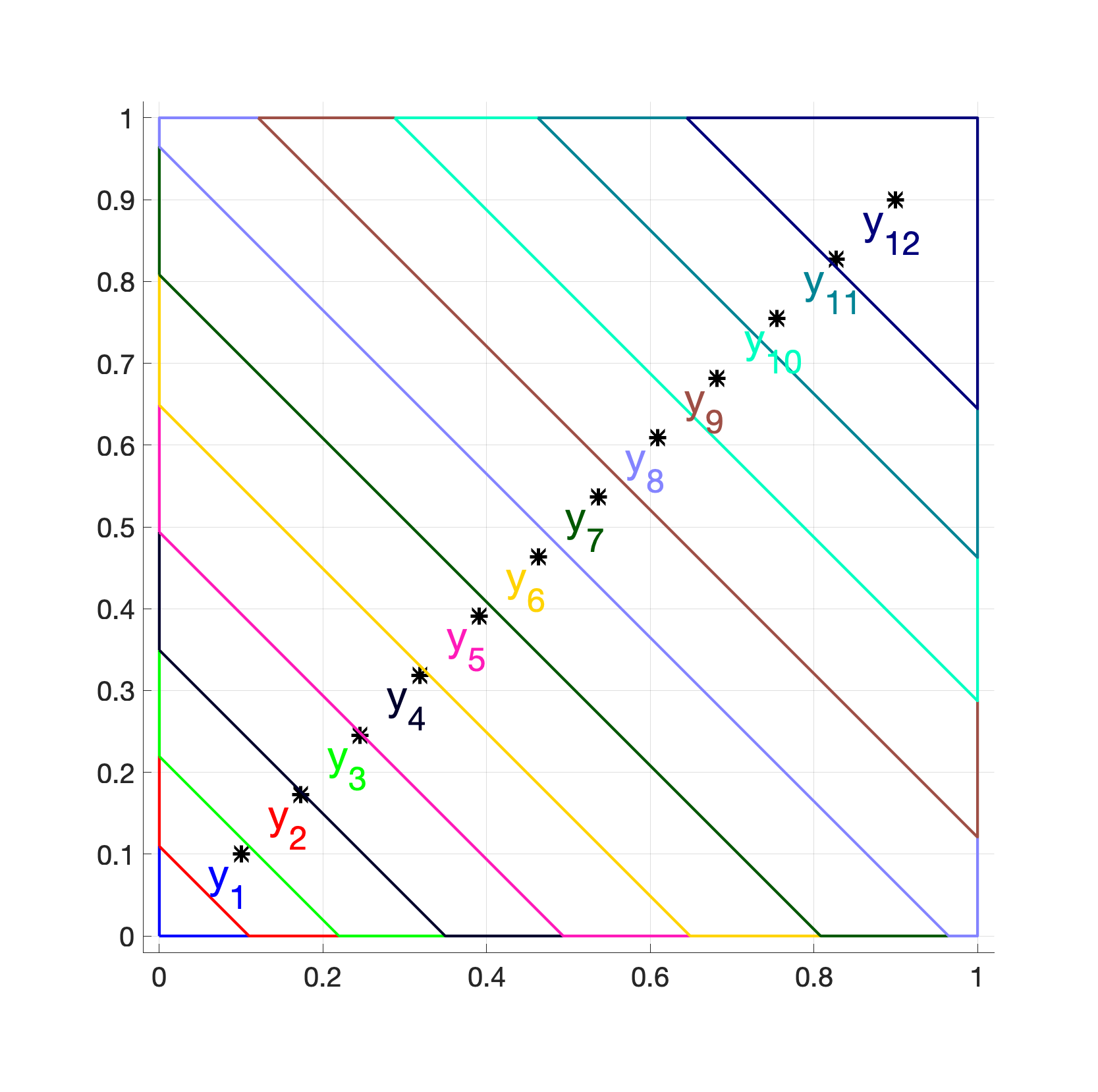}
        \caption{$d\mu_1 = dx_1dx_1$}
    \end{subfigure}
    \begin{subfigure}{.49\textwidth}
        \centering
        \includegraphics[width=\linewidth]{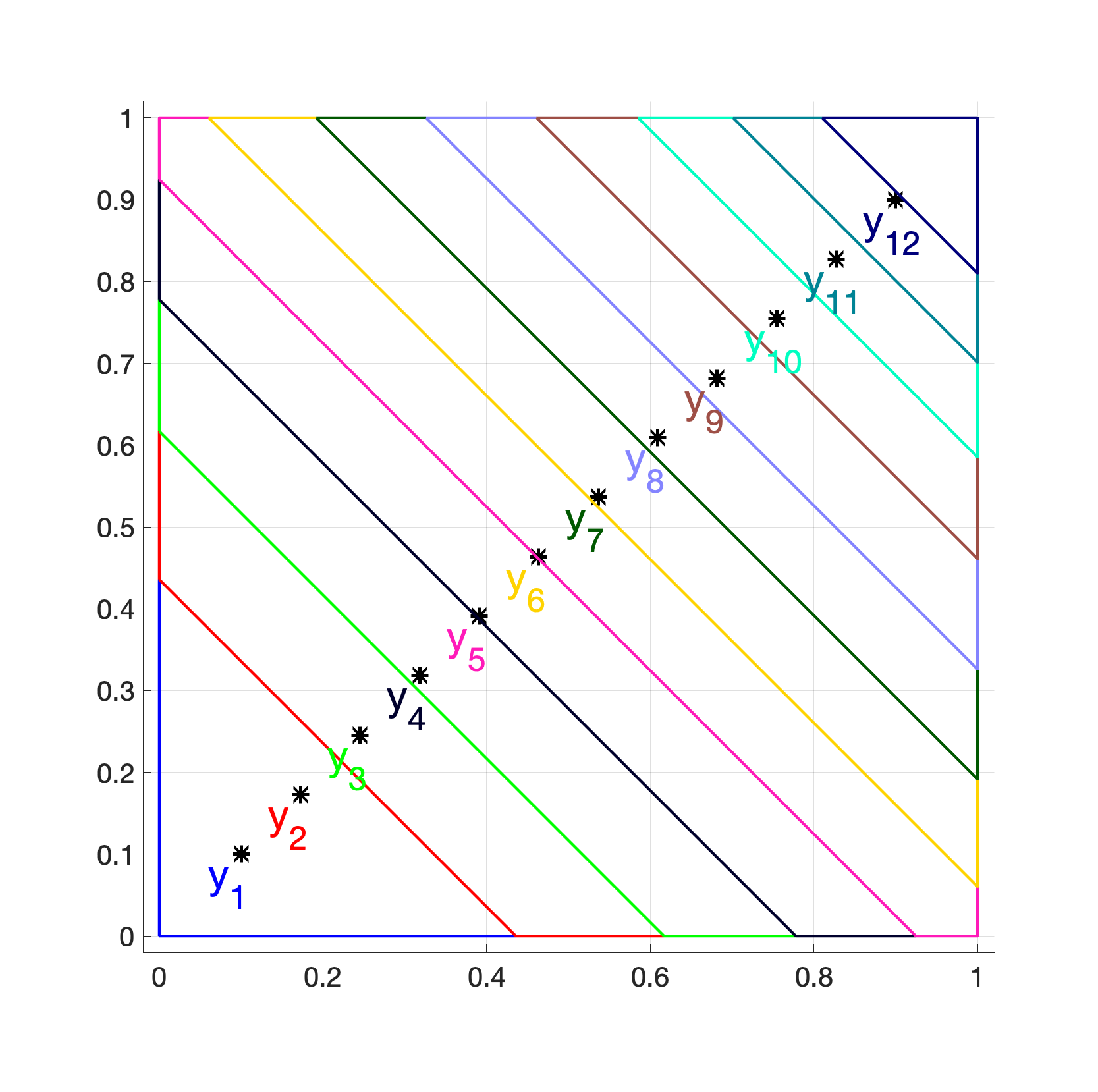}
        \caption{$d\mu_2 = 4x_1x_2dx_1dx_1$}
    \end{subfigure}
    \caption{Target points on a straight line $x_2 = x_1$, $N = 12$}
    \label{straight line}
\end{figure}

\begin{figure}[ht]
    \centering
    \begin{subfigure}{.49\textwidth}
        \centering
        \includegraphics[width=\linewidth]{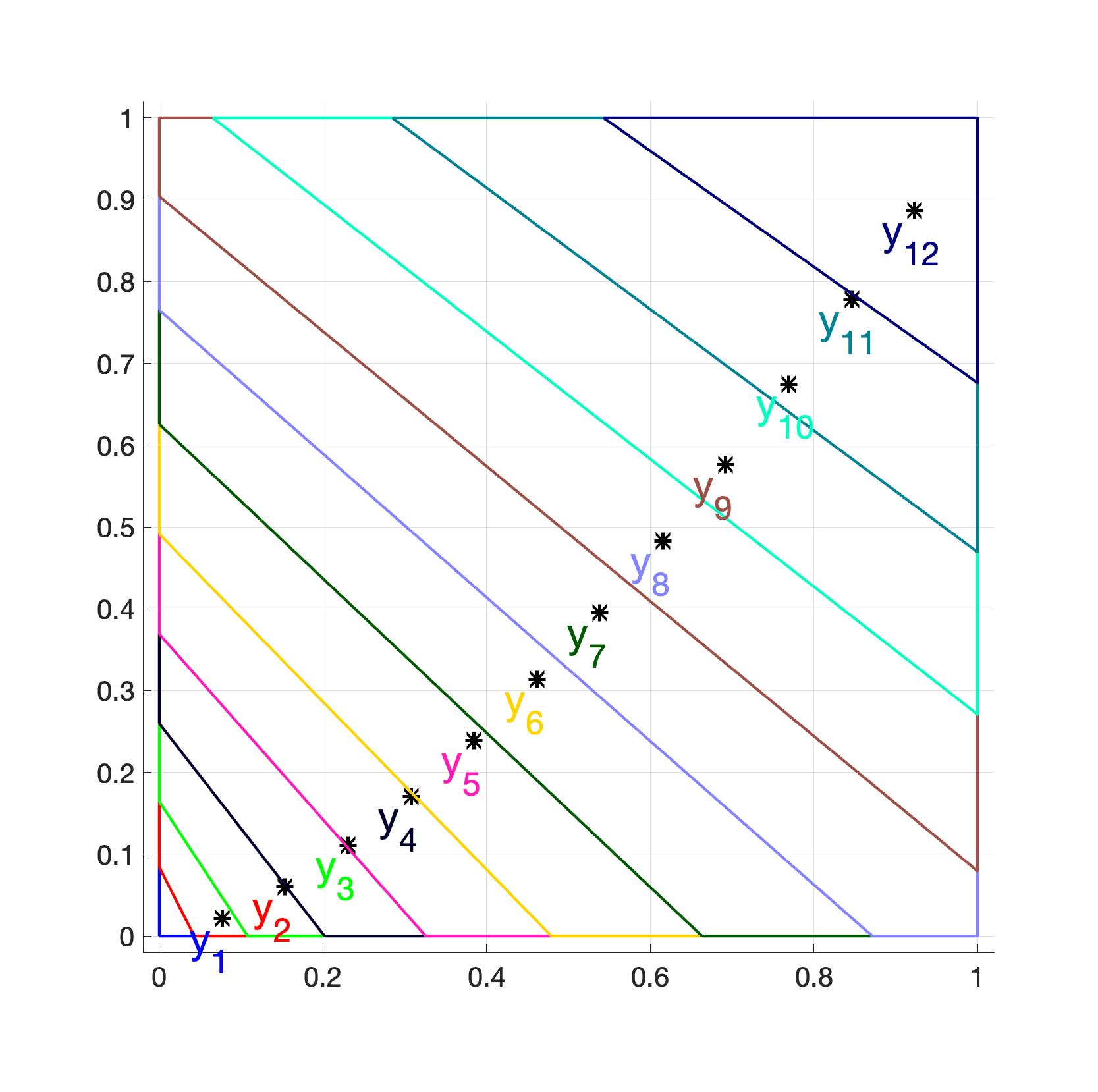}
        \caption{$d\mu_1 = dx_1dx_1$}
    \end{subfigure}
    \begin{subfigure}{.49\textwidth}
        \centering
        \includegraphics[width=\linewidth]{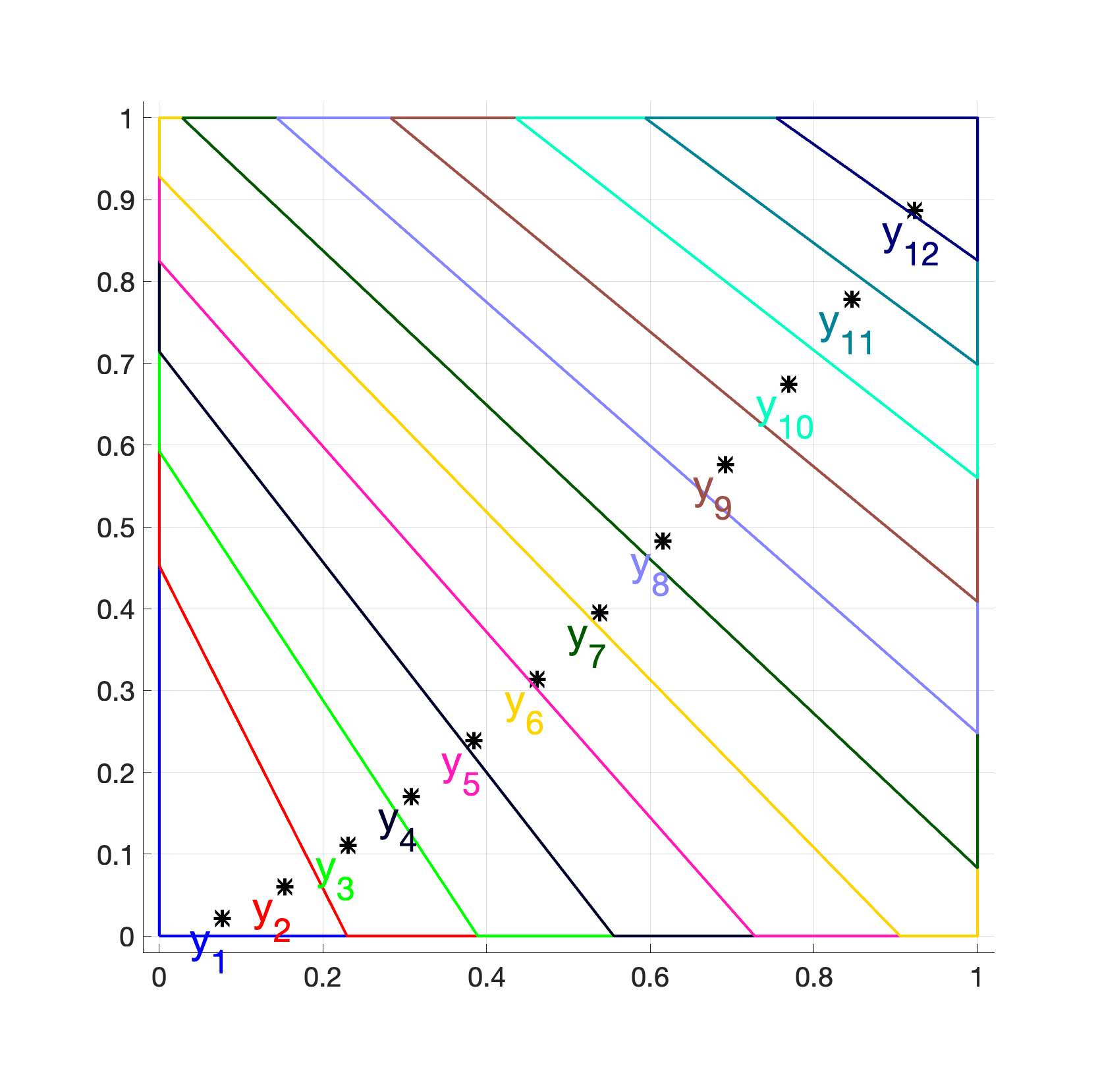}
        \caption{$d\mu_2 = 4x_1x_2dx_1dx_1$}
    \end{subfigure}
    \caption{Target points on a curve $x_2=x_1^{1.5}$, $N = 12$}
    \label{curve}
\end{figure}

\begin{figure}[ht]
    \centering
    \begin{subfigure}{.49\textwidth}
        \centering
        \includegraphics[width=\linewidth]{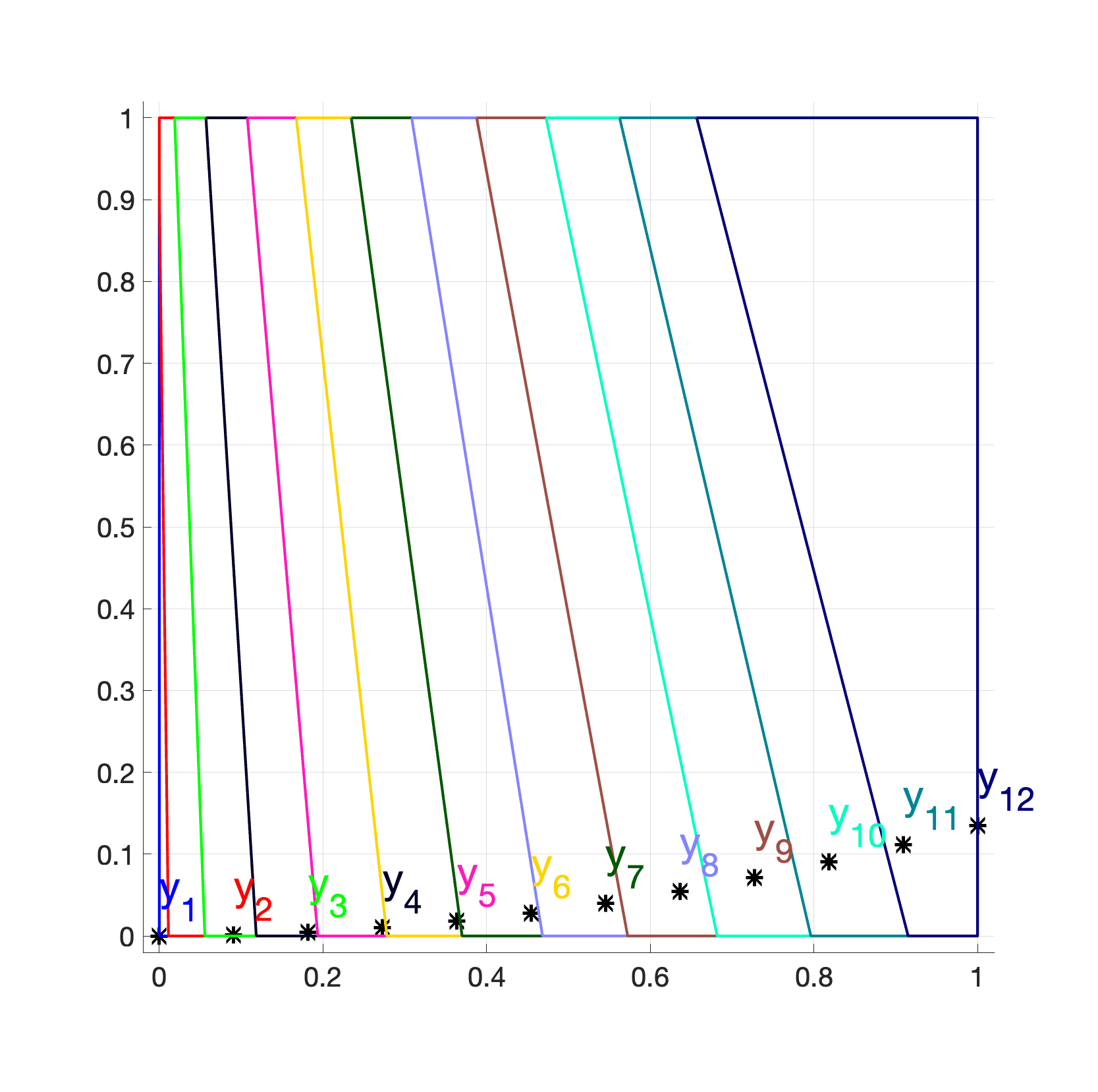}
        \caption{$d\mu_1 = dx_1dx_1$}
    \end{subfigure}
    \begin{subfigure}{.49\textwidth}
        \centering
        \includegraphics[width=\linewidth]{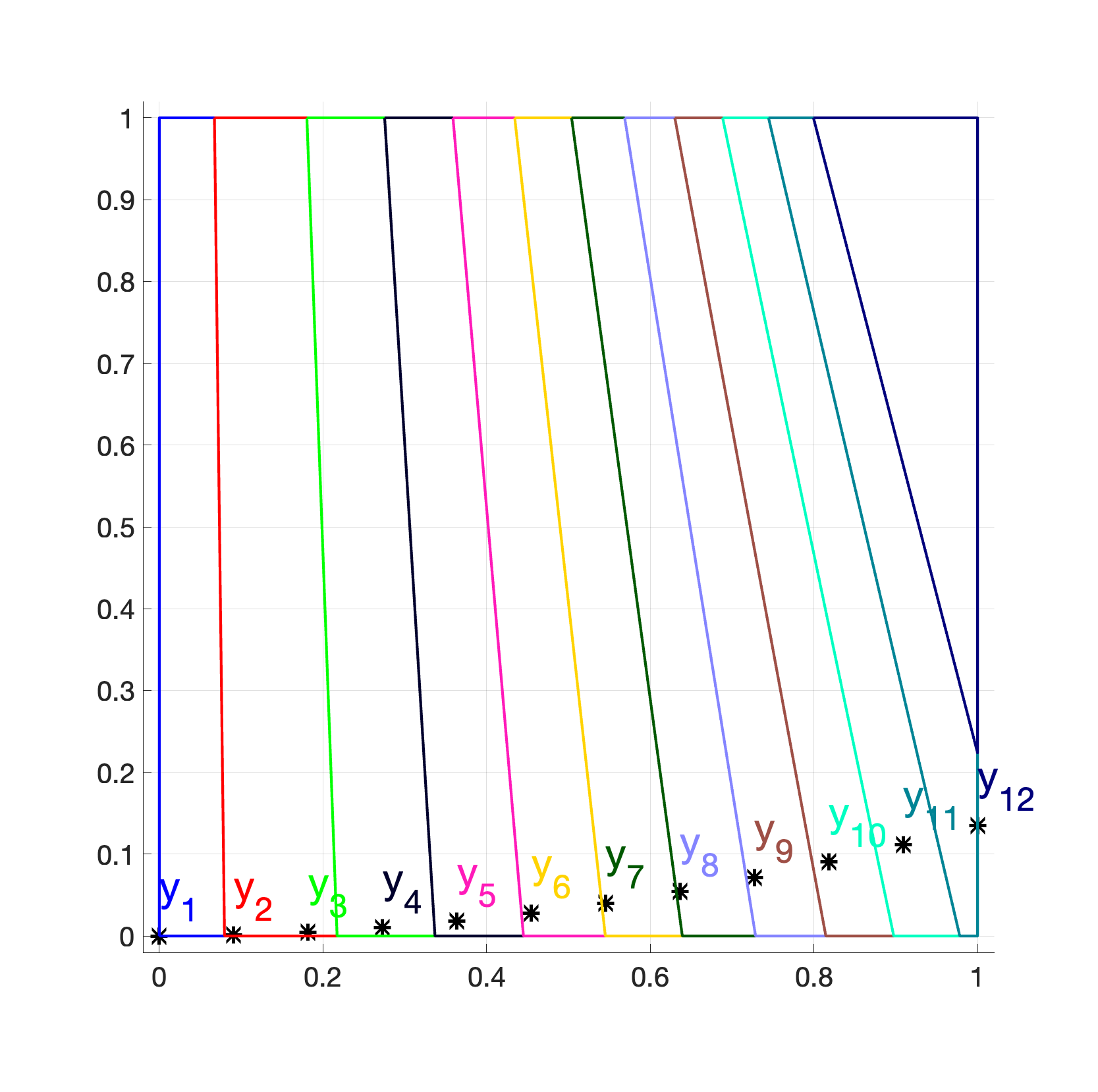}
        \caption{$d\mu_2 = 4x_1x_2dx_1dx_1$}
    \end{subfigure}
    \caption{Target points on a scaled parabola $x_2 = \big{(}\frac{x_1}{e}\big{)}^2$, $N=12$}
    \label{scaled parabola}
\end{figure}

%\bibliography{bibliography}{}
\bibliographystyle{plain}

% \FloatBarrier
% \bibliographystyle{plain}
% \bibliography{bibliography.bib}

\begin{thebibliography}{10}

\bibitem{halim2025multitoonedimensionalscreeningsemidiscrete}
Omar Abdul~Halim and Brendan Pass.
\newblock Multi-to one-dimensional screening and semi-discrete optimal transport, 2025.
\newblock Preprint currently available at https://arxiv.org/abs/2506.21740.

\bibitem{AguehCarlier2011}
M.~Agueh and G.~Carlier.
\newblock Barycenters in the {W}asserstein space.
\newblock {\em SIAM Journal on Mathematical Analysis}, 43(2):904--924, 2011.

\bibitem{blanchet2014Nash}
Adrien Blanchet and Guillaume Carlier.
\newblock From {N}ash to {C}ournot-{N}ash equilibria via the {M}onge–{K}antorovich problem.
\newblock {\em Philosophical Transactions of the Royal Society A}, 372(2028):20130398, 2014.

\bibitem{blanchet2014remarks}
Adrien Blanchet and Guillaume Carlier.
\newblock Remarks on existence and uniqueness of {C}ournot–{N}ash equilibria in the non-potential case.
\newblock {\em Mathematics and Financial Economics}, 8(4):417--433, 2014.

\bibitem{blanchet2016optimal}
Adrien Blanchet and Guillaume Carlier.
\newblock Optimal transport and {C}ournot-{N}ash equilibria.
\newblock {\em Math. Oper. Res.}, 41(1):125--145, 2016.

\bibitem{BlanchetCarlierNenna2017}
Adrien Blanchet, Guillaume Carlier, and Luca Nenna.
\newblock Computation of {C}ournot--{N}ash equilibria by entropic regularization.
\newblock {\em Vietnam Journal of Mathematics}, September 2017.

\bibitem{buttazzo2005model}
Giuseppe Buttazzo and Filippo Santambrogio.
\newblock A model for the optimal planning of an urban area.
\newblock {\em SIAM J. Math. Anal.}, 37(2):514--530, 2005.

\bibitem{buttazzo2009mass}
Giuseppe Buttazzo and Filippo Santambrogio.
\newblock A mass transportation model for the optimal planning of an urban region.
\newblock {\em SIAM Review}, 51(3):593--610, 2009.

\bibitem{CarlierEkeland2004}
Guillaume Carlier and Ivar Ekeland.
\newblock The structure of cities.
\newblock {\em Journal of Global Optimization}, 29(4):371--376, 2004.

\bibitem{carlier2005urban}
Guillaume Carlier and Filippo Santambrogio.
\newblock A variational model for urban planning with traffic congestion.
\newblock {\em ESAIM: Control, Optimisation and Calculus of Variations}, 11(4):595--613, 2005.

\bibitem{chiappori2010hedonic}
P.-A. Chiappori, R.~McCann, and L.~Nesheim.
\newblock Hedonic price equilibria, stable matching and optimal transport; equivalence, topology and uniqueness.
\newblock {\em Economic Theory}, 42(2):317--354, 2010.

\bibitem{chiappori2017multi}
Pierre-Andr{\'e} Chiappori, Robert~J. McCann, and Brendan Pass.
\newblock Multi- to one-dimensional optimal transport.
\newblock {\em Communications on Pure and Applied Mathematics}, 70(12):2405--2444, 2017.

\bibitem{COCV_2005__11_1_57_0}
Ivar Ekeland.
\newblock An optimal matching problem.
\newblock {\em ESAIM: Control, Optimisation and Calculus of Variations}, 11(1):57--71, 2005.

\bibitem{Galichon2016}
Alfred Galichon.
\newblock {\em Optimal Transport Methods in Economics}.
\newblock Princeton University Press, 2016.

\bibitem{McCann2001}
Robert~J. McCann.
\newblock Polar factorization of maps on riemannian manifolds.
\newblock {\em Geom. Funct. Anal.}, 11(3):589--608, 2001.

\bibitem{merigot2021optimal}
Quentin Merigot and Boris Thibert.
\newblock Optimal transport: discretization and algorithms.
\newblock In {\em Handbook of numerical analysis}, volume~22, pages 133--212. Elsevier, 2021.

\bibitem{nenna2020variational}
Luca Nenna and Brendan Pass.
\newblock Variational problems involving unequal dimensional optimal transport.
\newblock {\em Journal de Math{\'e}matiques Pures et Appliqu{\'e}es}, 139:83--108, 2020.

\bibitem{nenna2024note}
Luca Nenna and Brendan Pass.
\newblock A note on {C}ournot-{N}ash equilibria and optimal transport between unequal dimensions.
\newblock In {\em Optimal Transport Statistics for Economics and Related Topics}, pages 117--130. Springer, 2023.

\bibitem{PeyreCuturi2019}
Gabriel Peyré and Marco Cuturi.
\newblock Computational optimal transport.
\newblock {\em Foundations and Trends® in Machine Learning}, 11(5-6):355--607, 2019.

\bibitem{santambrogio2015optimal}
Filippo Santambrogio.
\newblock {\em Optimal Transport for Applied Mathematicians: Calculus of Variations, PDEs, and Modeling}, volume~87 of {\em Progress in Nonlinear Differential Equations and Their Applications}.
\newblock Birkhäuser, Cham, 2015.

\bibitem{Villani2009}
Cédric Villani.
\newblock {\em Optimal Transport: Old and New}, volume 338 of {\em Grundlehren der mathematischen Wissenschaften}.
\newblock Springer, 2009.

\end{thebibliography}

\clearpage
\appendix
\section{Appendix: Algorithm Pseudocodes}

\begin{algorithm}
		\caption{\textit{DampedNewton} -- Damped Newton's Method}\label{Alg_DampedNewton}
		\begin{algorithmic}
			\Require $\Omega$, $\mu(x)$, $\mathbf{v^0}$, $y_i$, $i=1,2\dots, N$, $\mathtt{MAXIT}$, $\mathtt{TOL}$
			\Comment{Default: $\mathbf{v^0} = \mathbf{0}$, $\mathtt{MAXIT=20}$, ${\mathtt{TOL}}=10^{-5}$}
			\State $\{err^0\}_{i} \gets \mu(Lag_i(\y,\mathbf{v^0})) - \frac{e^{-v^0_i}}{\sum_{k=1}^N e^{-v^0_k}}$,  $\forall i$
			\For{$j \gets 0$ to $\mathtt{MAXIT}$}
    			\State $H^j \gets \nabla G(\mathbf{v}^{(j)})$\Comment{Centered Finite Difference scheme}
    			\State $\mathbf{s} \gets - (H^j)^{\dagger}\mathbf{err}^{j}$, $\mathbf{v}^{j+1} \gets \mathbf{v}^{j} + \mathbf{s}$
    			\While{$\text{Lag}(y_{i-1}) \cap \text{Lag}(y_{i}) \cap \text{Lag}(y_{i+1}) \neq \emptyset, \quad i=2,3,\dots, N-1$} \Comment{Damping if not nested}
                    \State $\mathbf{s} \gets \frac{1}{2}\mathbf{s}$, $\mathbf{v}^{j+1} \gets \mathbf{v}^{j} + s$
    			\EndWhile
    			\State $\{err^{j+1}\}_{i} \gets \mu(Lag_i(\y,\mathbf{v^{j+1}})) - \frac{e^{-v^{j+1}_i}}{\sum_{k=1}^N e^{-v^{j+1}_i}}$,  $\forall i$
    			\If{$||\mathbf{err}^{j+1}||_{\infty}\leq \mathtt{TOL}$}
                    \State \Return $\mathbf{v^{j+1}}$
    			\EndIf
			\EndFor
	\end{algorithmic}
\end{algorithm}

\begin{algorithm}
		\caption{\textit{ErrorFunc} --Error Computation}\label{Alg_ErrComp}
		\begin{algorithmic}
			\Require $\{\Omega,\ \mu(x)\}$, $y_i$, $i=1,2,\dots, N$, $C$
            \State $v_1 \gets 0$
			\For{$j \gets 2$ to $N$}
                \State $\mu(\text{Lag}_{j-1}(v_1,\dots, v_j) - e^{C^{*}-v_{j-1}} = 0 \quad \Rightarrow \quad v_j$
                \If{$e^{C^{*}-v_{j}} - \mu(\text{Lag}_{j}(v_1,\dots, v_j)>0$}
                    \Return $Error \gets NAN $ \Comment{Remark \ref{Nest_Issue2}}
                \EndIf
            \EndFor
            \State $Error \gets \mu(\text{Lag}_{N}(v_1,\dots, v_N) - e^{C^{*}-v_{N}}$
		\end{algorithmic}
\end{algorithm}

\begin{algorithm}
		\caption{\textit{NestBisection} --Nested Bisection on $C$ Value}\label{Alg_NestBisect}
		\begin{algorithmic}
			\Require $\{\Omega,\ \mu(x)\}$, $y_i$, $i=1,2,\dots, N$, $C^0$, $C^1$, $\mathtt{TOL}$ \Comment{Default: $\mathtt{TOL}= 10^{-5}$}
            \While{$|Error|>\mathtt{TOL}$}
    			\State $C^{*} \gets \frac{1}{2}(C^0 + C^1)$
    			\State $Error \gets ErrorFunc(C^{*})$ \Comment{Algorithm \ref{Alg_ErrComp}}
                \If{$Error >0$}
                    \State $C^0 \gets C^{*}$ 
                \ElsIf{$Error >0 \lor Error = NAN$}
                    \State $C^1 \gets C^{*}$
                \EndIf
            \EndWhile
		\end{algorithmic}
\end{algorithm}

\begin{algorithm}[H]
		\caption{\textit{NestNewton} --Nested Newton on $C$ Value}\label{Alg_NestNewt}
		\begin{algorithmic}
			\Require $\{\Omega,\ \mu(x)\}$, $y_i$, $i=1,2,\dots, N$, $C^0$, $\mathtt{TOL}$ \Comment{Default: $\mathtt{TOL}= 10^{-5}$}
            \State $h \gets \mathtt{eps}^{\frac13}$, $Error \gets ErrorFunc(C^{0})$ \Comment{Algorithm \ref{Alg_ErrComp}}
            \While{$|Error|>\mathtt{TOL}$}
            \State $H \gets \frac{ErrorFunc(C^{0}+h)-ErrorFunc(C^{0}-h)}{2h}$ \Comment{Centered Finite Difference Scheme}
            \State $s \gets - \frac{Error}{H}$, $C^{1} \gets C^0 + s$
            \While{$C^{1} \geq 0$} \Comment{Solution is always negative}
            \State $s \gets \frac{1}{2}s$, $C^{1} \gets C^0 + s$
    		\EndWhile
            \State $Error \gets ErrorFunc(C^1)$
            \While{$Error = NAN$} \Comment{Remark \ref{Nest_Issue2}}
            \State $s \gets \frac{1}{2}s$, $C^{1} \gets C^0 + s$, $Error \gets ErrorFunc(C^1)$
    		\EndWhile
            \State $C^0 \gets C^1$, 
            \EndWhile
		\end{algorithmic}
\end{algorithm}

\section{Convergence of the Nested Algorithm of the Internal Energy Problem}\label{convergence}
While  Subsection \ref{DiscContinuous} focused on the numerical implementation, we now shift to a more theoretical perspective. Accordingly, we return to the original notation used in Section \ref{nestedness}, which is better suited for analytical arguments. In this section, we will prove the convergence of  the numerical algorithm introduced in Subsection \ref{DiscContinuous} for computing the solution of \eqref{cong} under the assumption that the solution exhibits discrete nestedness. 

We begin by describing the algorithm in full detail.  It  aims to find a measure $\nu \in \mathcal{P}(Y)$ that satisfies the first-order optimality condition \eqref{optcon} associated with \eqref{cong}. To solve this system, we exploit the correspondence between $\nu$ and $v$. Specifically, we express $\nu_i$ as
$\nu_i = (f')^{-1}(C - v_i),$
and associate each $\nu_i$ with a subset $X_i(v) \subset X$ such that $\mu(X_i(v)) = \nu_i$.
%{\color{red} BP: I suggest the notation $\nu_i^C$ for the components of $\nu^C$.}
For each chosen $ C $, we determine $ v^C = (v^C_i)_{i=1}^N $ using the relation between $ \nu_i^C = \mu(X_i(v^C)) $ and $ (f')^{-1}(C - v^C_i) $. We then define an error function $h(C)$ which we will prove is monotone. Consequently, the solution can be computed using the bisection method to solve $ h(C) = 0 $.

We define $h(C)$ as follows:

For each $C$, we set $v^C_1 = 0$ and define $X_1(v^C) = X^N_{\ge}(y_1, k_1)$ such that
\[
\mu(X^N_{\ge}(y_1, k_1)) = \min\left\{1, (f')^{-1}(C - v^C_1)\right\}.
\]
Then, for each $i \geq 2$, we set $v^C_i = v^C_{i-1} + k_{i-1}$ and define $\overline{k}_i$ such that
\[
\mu(X^N_{\ge}(y_i, \overline{k}_i)) = (f')^{-1}(C - v^C_i)+\sum_{j=1}^{i-1}\mu(X_j(v^C)).
\]
If $\overline{k}_i < k_{\max}(y_{i-1}, k_{i-1})$, we set $k_i = \overline{k}_i$; otherwise, we set $k_i = k_{\max}(y_{i-1}, k_{i-1})$. We then define
\[
X_i(v^C) = X^N_{\ge}(y_i, k_i) \setminus X^N_{\ge}(y_{i-1}, k_{i-1}).
\]

This process continues until we reach some $m^C=\max\{i:1-\mu(X_\ge^N(y_{i-2},k_{i-2}))\ge (f')^{-1}(C-v_{i-1}^C)\} \leq N,$ 
at which point we set
\(
X_{m^C}(v^C) = X \setminus X^N_{\ge}(y_{m^C-1}, k_{m^C-1}),
\quad \text{and} \quad
X_i(v^C) = \emptyset \quad \text{for all } i > m^C.
\) If $m^C<N,$ we set $h(C)=-\infty;$ otherwise, we set $h(C)=1-\mu(X^N_\ge(y_{N-1},k_{N-1}))-(f')^{-1}(C - v^C_N).$ Also, we set $\nu_i^C=\mu(X_i(v^C))$ for all $1\le i\le m^C$ and $\nu^C=\sum_{i=1}^{m^C}\nu_i^C\delta_{y_i}.$
%{\color{red} BP: What do we mean by whenever $k_i$ exists? It looks like it always exists, and the inequality always holds...}

\begin{rem}
    If $k_i < k_{\max}(y_{i-1}, k_{i-1})$, then $\nu_i^C = \mu(X_i(v^C)) = (f')^{-1}(C - v^C_i)$. Otherwise, $\nu_i^C = \mu(X_i(v^C)) \geq (f')^{-1}(C - v^C_i)$.
\end{rem}

\begin{lem}\label{lem: monotonicity of error function}
    The function $h(C)$ is decreasing in $C.$
\end{lem}

\begin{proof}
Let $ C < C' $ and let $ v^C = (v^C_i)_{i=1}^N $ and $ v^{C'} = (v^{C'}_i)_{i=1}^N $ be associated with $ (X_i(v^C))_{i=1}^N $ and $ (X_i(v^{C'}))_{i=1}^N $, respectively. We adapt $v^{C'}_i$ and $k_i'$ such that they are defined similar to $v^C_i$ and $k_i$ respectively, when $C$ is replaced by $C'.$ 

\begin{comment}
By the strict monotonicity of $ (f')^{-1} $, we obtain  
\[
\mu(X_1(v^C)) = (f')^{-1}(C - v^C_1)=\min\{1,(f')^{-1}(C - v^C_1)\} < \min\{1,(f')^{-1}(C' - v^C_1')\} = \mu(X_1(v^{C'}))
\]
as $ v^C_1 = v^C_1' = 0 $, otherwise,
\[\mu(X_1(v^C))=1<(f')^{-1}(C - v^C_1)<(f')^{-1}(C' - v^C_1')\] and so $\mu(X_1(v^C))=\mu(X_1(v^{C'}))=1$ which implies $h(C)=h(C')=-\infty.$ Since  
\[
\mu(X_{\ge}(y_1, k_1)) = \mu(X_1(v^C)) < \mu(X_1(v^{C'})) = \mu(X_{\ge}(y_1, k_1'))
\]
and $ \mu(X_{\ge}(y_1, k)) $ is a decreasing function, we conclude that $ k_1 > k_1' $. 
\end{comment}

By the strict monotonicity of $(f')^{-1}$, we obtain  
\[
\mu(X_1(v^C)) = (f')^{-1}(C - v^C_1) = \min\{1, (f')^{-1}(C - v^C_1)\} < \min\{1, (f')^{-1}(C' - v^{C'}_1)\} = \mu(X_1(v^{C'}))
\]
in the case where $(f')^{-1}(C - v^C_1) < 1$ since $v^C_1 = v^{C'}_1 = 0$. This yields  
\[
\mu(X^N_\ge(y_1,k_1))=\mu(X_1(v^C)) < \mu(X_1(v^{C'}))=\mu(X^N_\ge(y_1,k_1')),
\]
and since $\mu(X^N_{\ge}(y_1, k))$ is a decreasing function of $k$, this implies $k_1 > k_1'$.

On the other hand, if $(f')^{-1}(C - v^C_1) \geq 1$, then
$\mu(X_1(v^C)) = \mu(X_1(v^{C'})) = 1,$
and hence $h(C) = h(C') = -\infty$, so the comparison of $k_1$ and $k_1'$ becomes irrelevant in this case.

Proceeding to the case $(f')^{-1}(C - v^C_1) < 1$, we have  $v^C_2 = k_1 + v^C_1 > k_1' + v^{C'}_1 = v^{C'}_2.$

We now prove that $ k_{i+1} > k_{i+1}' $ and $ v^C_{i+2} > v^{C'}_{i+2} $ whenever $ k_i > k_i' $ and $ v^C_{i+1} > v^{C'}_{i+1} $. Suppose that $ k_{i+1} \leq k_{i+1}' $, then $X^N_\ge(y_{i+1},k_{i+1}')\subseteq X^N_\ge(y_{i+1},k_{i+1})$ and so  
\[
\begin{array}{ll}
\mu(X_{i+1}(v^C)) &= \mu(X^N_{\ge}(y_{i+1}, k_{i+1}) \setminus X^N_{\ge}(y_i, k_i))  \vspace{1pt}\\
&\geq \mu(X^N_{\ge}(y_{i+1}, k_{i+1}') \setminus X^N_{\ge}(y_i, k_i'))
= \mu(X_{i+1}(v^{C'})) \geq (f')^{-1}(C' - v^{C'}_{i+1}) > (f')^{-1}(C - v^C_{i+1}),
\end{array}
\]
since $ C' - v^{C'}_{i+1} > C - v^C_{i+1} $. Consequently,  $k_{i+1} = k_{\max}(y_i, k_i).$
Since $ k_i > k_i' $, we have  
\[
X^N_{\ge}(y_i, k_i) \subset X^N_{\ge}(y_i, k_i') \subseteq X^N_{\ge}(y_{i+1}, k_{\max}(y_i, k_i')) \subseteq X^N_{\ge}(y_{i+1}, k_{i+1}'),
\]
which implies that $k_{i+1}=k_{\max}(y_i,k_i) \geq k_{i+1}'.$ Thus, $ k_{i+1}' = k_{\max}(y_i, k_i)=k_{i+1} $.  

We claim that there exists $ x \in \overline{X^N_=(y_i, k_i)} \cap \overline{X^N_=(y_{i+1}, k_{\max}(y_i, k_i))} $. Consider a sequence $ k^n $ that decreases to $ k_{i+1} $. By the definition of $ k_{\max}(y_i, k_i) $, we have $ X^N_{\ge}(y_i, k_i) \not\subset X^N_>(y_{i+1}, k^n) $, so we take $ x_n \in X^N_{\ge}(y_i, k_i)  \setminus X^N_>(y_{i+1}, k^n) $. Then,  
\[
c(x_n, y_{i+1}) - c(x_n, y_i) \geq k_i\quad
\text{ and }  
\quad c(x_n, y_{i+2}) - c(x_n, y_{i+1}) \leq k^n.
\]
Taking a subsequential limit $ \overline{x} \in \overline{X} $ of $(x_n)$ gives  
\[
c(\overline{x}, y_{i+1}) - c(\overline{x}, y_i) \geq k_i\quad
\text{ and }\quad c(\overline{x}, y_{i+2}) - c(\overline{x}, y_{i+1}) \leq k_{\max}(y_i, k_i).
\]
Since $ X^N_{\ge}(y_i, k_i) \subset X^N_\ge(y_{i+1}, k_{\max}(y_i, k_i)) $, we conclude that  $\overline{x} \in \overline{X^N_=(y_i, k_i)} \cap \overline{X^N_=(y_{i+1}, k_{\max}(y_i, k_i))},$ proving our claim.  

 Since $\overline{x}\in \overline{X^N_=(y_i,k_i)}\subset \overline{X^N_>(y_i,k_i')}\setminus\overline{X^N_=(y_i,k_i')}$, we have $c(\overline{x},y_{i+1})-c(\overline{x},y_i)>k_i'$ and by the continuity of $ c(\cdot, y_{i+1}) - c(\cdot, y_i) $ at $ \overline{x} $, there exists a neighborhood $ U $ of $ \overline{x} $ such that  $c(x, y_{i+1}) - c(x, y_i) > k_i'$
for all $ x \in U\cap X $. Also, by \ref{nonzero1} and the fact $\overline{x}\in \overline{X^N_=(y_{i+1},k_{\max}(y_i,k_i))}$, there exists $ x \in U \cap X\subseteq X^N_\ge(y_i,k_i') $ such that  $
c(x, y_{i+2}) - c(x, y_{i+1}) < k_{\max}(y_i, k_i),$
which contradicts $ X^N_{\ge}(y_i, k_i') \subseteq X^N_{\ge}(y_{i+1}, k_{\max}(y_i, k_i)) $. Hence, $ k_{i+1}' \neq k_{\max}(y_i, k_i) $, so  $k_{i+1}' < k_{i+1}$ 
and  $v^C_{i+2} = k_{i+1} + v^C_{i+1} > k_{i+1}' + v^{C'}_{i+1} = v^{C'}_{i+2}.$

 Let $C_0\in \Big\{C \mid \sup\{i \mid \mu(X_i(v^C)) > 0\} = N\Big\},$ then for all $1\le i\le N-2$ we have 
\[
1 - \mu(X^N_{\ge}(y_i, k_i)) - (f')^{-1}(C_0 - v^C_{i+1}) > 0.
\] 
Since $ v^C_{i} $ and $ k_i $ decrease as $ C $ increases, $ 1 - \mu(X^N_{\ge}(y_{i}, k_{i})) $ decreases and $ (f')^{-1}(C - v^C_{i+1}) $ increases, which implies  $1 - \mu(X^N_{\ge}(y_i, k_i)) - (f')^{-1}(C - v^C_{i+1}) > 0$
decreases for all $1\le i\le N-2.$ In particular,  we have $1 - \mu(X^N_{\ge}(y_{N-2}, k_{N-2})) - (f')^{-1}(C - v^C_{N-1}) > 0$ for all $C\le C_0$ and so $\mu(X_N(v^C))=1-\mu(X^N_{\ge}(y_{N-1}, k_{N-1}))>0$ for all $C\le C_0.$ We conclude that 
\[
(-\infty,C^*)\subseteq\Big\{C \mid \sup\{i \mid \mu(X_i(v^C)) > 0\} = N\Big\} \subseteq (-\infty, C^*]
\]
for $C^*=\sup\Big\{C \mid \sup\{i \mid \mu(X_i(v^C)) > 0\} = N\Big\}.$

When $ C > C^* $, we have $ h(C) = -\infty $. When $ C = C^* $, then
\[
h(C) =
\begin{cases}
-\infty & \text{if } m^C<N %1 - \mu(X^N_{\ge}(y_{N-2}, k_{N-2})) - (f')^{-1}(C - v^C_{N-1}) < 0, 
\\
1-\mu(X^N_\ge(y_{N-1},k_{N-1}))- (f')^{-1}(C - v^C_{N}) & \text{if } m^C=N %1 - \mu(X^N_{\ge}(y_{N-2}, k_{N-2})) - (f')^{-1}(C - v^C_{N-1}) \ge 0.
\end{cases}
\] where $m^C$ is defined as above.
When $ C < C^* $, since $ v^C_i $ and $ k_i $ decrease as $ C $ increases, the quantity $ 1 - \mu(X^N_{\ge}(y_{N-1}, k_{N-1})) $ decreases, while $ (f')^{-1}(C - v^C_N) $ increases. This implies that $ h(C) $ is decreasing, completing the proof.
\end{proof}

\begin{rem}\label{c range}
   By Theorem \ref{bounds} we obtain  
\[
(f')^{-1}(J_{-M_c |y_i-y_1|}(1) - M_c |y_i-y_1|) \leq \nu_i^C = (f')^{-1}(C - v^C_i) \leq (f')^{-1}(J_{M_c |y_i-y_1|}(1) + M_c |y_i-y_1|),
\]
which implies  
\[
J_{-M_c |y_i-y_1|}(1) - M_c |y_i-y_1| \leq C - v^C_i \leq J_{M_c |y_i-y_1|}(1) + M_c |y_i-y_1|.
\]
When $ i = 2 $, we get  
\[
J_{-M_c |y_2-y_1|}(1) - M_c |y_2-y_1| \leq C \leq J_{M_c |y_2-y_1|}(1) + M_c |y_2-y_1|,
\]
as $ v^C_1 = 0 $, which provides bounds on $ C $ that can be used to find $ C $ using the bisection method on $ h(C) = 0 $.

\end{rem}

\begin{rem}
    In this theoretical algorithm, the issue discussed in the last bullet point of Remark~\ref{Nest_Limit}—namely, that certain values of $C$ may produce a Laguerre tessellation that is not nested—is resolved within this framework. In such cases, nestedness can be enforced using the correction described above: whenever
\[
\overline{k}_i=v_{i+1}^* - v_i^* > k_{\max}(y_{i-1}, v_i^* - v_{i-1}^*),
\]
we replace $v_{i+1}^*$ by
\[
v_{i+1}^* := v_i^* + k_{\max}(y_{i-1}, v_i^* - v_{i-1}^*),
\]
where the notation $v^*$ refers to the updated numerical values, as also used in Remark~\ref{Nest_Issue2}. However, in all the examples shown in Subsection~\ref{numerical_examples}, this situation does not arise, and hence this correction step is omitted in the implementation.
\end{rem}
 Lemma \ref{lem: monotonicity of error function} implies that if there is some $C$ such that $h(C)=0$, the bisection algorithm will converge to this $C$.  The proposition below, in turn, verifies that if the solution is discretely nested, the $C$ corresponding to the solution satisfies $h(C)=0$.  Together, these two results ensure that if the solution is discretely nested, the bisection algorithm applied to the error function $h$ constructed above will converge to it. 
\begin{prop}
    Let $\nu$ be the solution of \eqref{cong}. If $(c,\mu,\nu)$ is discretely nested, then $\nu_i=\nu_i^C=\mu(X_i(v^C))$ for all $1\le i\le N$ where $C$ satisfies $h(C)=0.$
\end{prop}
%{\color{red}BP: I don't think we use the notation $\nu(y_i)$ elsewhere.  It probably makes more sense to use $\nu_i$, as then another symbol, like $\nu_i^C$, for the potentially different value below}
\begin{proof}
%Suppose that $ C $ is a root of $ h $. Then,  
%$ h(C) = \mu(X_N(v^C)) - (f')^{-1}(C - v_N) = 0. $  
Since $ \nu $ is the solution of \eqref{cong}, $ \nu $ satisfies  
$ \nu_i = (f')^{-1}(C - v_i) $  
for all $ 1 \leq i \leq N $, where $ v $ is the Kantorovich potential of $ (c, \mu, \nu) $ with $ v_1 = 0 $ for some $C$.   Let $\nu_i^C=\mu(X_i(v^C))$ and we claim that $\nu_i^C=\nu_i.$
From the discrete nestedness of $ (c, \mu, \nu) $, we have  
$ \mu(X^N_{\ge}(y_1, v_2 - v_1)) = \nu_1.$  Then,  
$ \mu(X_1(v^C))=\mu(X^N_{\ge}(y_1, v_2 - v_1)) = (f')^{-1}(C) = \nu_1. $ which implies $ v_2 - v_1 = k_1 $, and so  
$ v^C_2 = k_1 + v_1 = v_2. $  
Also, from  the fact
$ \nu_2 = (f')^{-1}(C - v_2), $  
and the discrete nestedness, we get  
$ \mu(X^N_{\ge}(y_2, v_3 - v_2)) = \nu_1 + \nu_2, $  
and  
$ v_3 - v_2 < k_{\max}(y_1, v_2 - v_1). $  
As $ v^C_2 = v_2 $, we conclude that  
$ \nu_2^C = (f')^{-1}(C - v^C_2) $  
and  
$ \overline{k_2} = v_3 - v_2 < k_{\max}(y_1, k_1) $  
and we set $ k_2 = \overline{k}_2 $ where  
$ \mu(X^N_{\ge}(y_2, \overline{k}_2)) = \nu_1^C + \nu_2^C. $  
So, $ v_3 - v_2 = k_2 $ and $ v_3 = k_2 + v^C_2 = v^C_3. $

Proceeding inductively, we get $ v^C_i = v_i $ and $ \nu_i^C = \nu_i $ for all $ 1 \leq i \leq N $ and $h(C)=1-\mu(X^N_\ge(y_{N-1},k_{N-1}))-(f')^{-1}(C-v^C_N)=\nu_N-(f')^{-1}(C-v^C_N)=0$, which completes the proof.

 \end{proof}

\end{document}